\documentclass[12pt]{article}

\usepackage[margin=1.25in]{geometry}

\usepackage{amssymb,amsmath}
\usepackage{amsthm}
\usepackage{hyperref}
\usepackage{mathrsfs}
\usepackage{textcomp}
\hypersetup{
    colorlinks,%
    citecolor=black,%
    filecolor=black,%
    linkcolor=black,%
    urlcolor=black
}

\usepackage{verbatim}

\usepackage{sectsty}

\makeatletter

\newdimen\bibspace
\renewenvironment{thebibliography}[1]{%
 \section*{\refname 
       \@mkboth{\MakeUppercase\refname}{\MakeUppercase\refname}}%
     \list{\@biblabel{\@arabic\c@enumiv}}%
          {\settowidth\labelwidth{\@biblabel{#1}}%
           \leftmargin\labelwidth
           \advance\leftmargin\labelsep
           \itemsep\bibspace
           \parsep\z@skip     %
           \@openbib@code
           \usecounter{enumiv}%
           \let\p@enumiv\@empty
           \renewcommand\theenumiv{\@arabic\c@enumiv}}%
     \sloppy\clubpenalty4000\widowpenalty4000%
     \sfcode`\.\@m}
    {\def\@noitemerr
      {\@latex@warning{Empty `thebibliography' environment}}%
     \endlist}

\makeatother

\makeatletter

\newtheorem{thm}{Theorem}[section]
\newtheorem{lem}[thm]{Lemma}
\newtheorem{prop}[thm]{Proposition}
\newtheorem{defn}[thm]{Definition}

\newtheorem{cor}[thm]{Corollary}
\newtheorem{rem}[thm]{Remark}

\def\XXint#1#2#3{{\setbox0=\hbox{$#1{#2#3}{\int}$}
  \vcenter{\hbox{$#2#3$}}\kern-.5\wd0}}

                \newcommand{\lda}{\lambda}
\newcommand{\om}{\Omega}                \newcommand{\pa}{\partial}
\newcommand{\va}{\varepsilon}           \newcommand{\ud}{\mathrm{d}}
\newcommand{\be}{\begin{equation}}      \newcommand{\ee}{\end{equation}}

\newcommand{\R}{\mathbb{R}}

\begin{document}

\title{\textbf{Stability of the separable solutions for a nonlinear boundary diffusion problem}
\bigskip}

\author{\medskip  Tianling Jin\footnote{T. Jin is partially supported by Hong Kong RGC grants GRF 16306320 and GRF 16303822, and National Natural Science Foundation of China grant 12122120.}, \  \
Jingang Xiong\footnote{J. Xiong is partially supported by National Natural Science Foundation of China grants 11922104 and 11631002.}, \  \  Xuzhou Yang     }

\date{\today}

\maketitle

\begin{abstract}
In this paper, we study a nonlinear boundary diffusion equation of porous medium type arising from a boundary control problem.  We give a complete and sharp characterization of the asymptotic behavior of its solutions, and prove the stability of its separable solutions.
\end{abstract}


\smallskip


\bigskip

\textbf{Declarations of interest}: none.

\tableofcontents

\section{Introduction}

Let $ \Omega \subset \R^n$, $n \geq 2 $,  be a bounded smooth domain and $0<p<\infty$ be a real number. We consider the nonlinear boundary diffusion problem
\begin{equation}\label{1}
 \left\{
\begin{aligned}
  \Delta u  &=   0  \quad  \mbox{in }     \Omega\times (0,\infty),     \\
   \partial_{t}u^{p}&=   -  \partial_{\nu}u  -  a u     \quad  \mbox{on }     \partial \Omega\times (0,\infty),
\end{aligned}
\right.
\end{equation}
with the initial condition
\be \label{eq:initial}
u(\cdot,0) =   u_{0}(\cdot) > 0       \quad  \mbox{on }     \partial \Omega,
\ee
where $\Delta $ is the Laplace operator with respect to the spatial variable $x\in \om$, $\nu$ is the unit outer normal to $\pa \om$, $\pa_{\nu}$ is the outer normal derivative, $a\in C^\infty(\pa \om)$, and $u_0 \in C^\infty(\pa \om) $ is a positive function. This nonlinear boundary  diffusion problem arises from a boundary control problem; see Duvaut-Lions \cite{DL1972} and  Athanasopoulos-Caffarelli \cite{2010Continuity}. When $n\ge 3$, $p=\frac{n}{n-2}$ and $a=-\frac{n-2}{2(n-1)}H_{\om}$ with $H_{\om}$  being the mean curvature of $\om$, it is the unnormalized boundary Yamabe flow; see  Brendle \cite{brendle2002generalization} and Almaraz \cite{Alm}. If we denote $\mathscr{B}:  H^{\frac12}(\partial \Omega) \rightarrow H^{-\frac12}(\partial \Omega)$ as the Dirichlet to Neumann map, that is, for any $f \in H^{\frac12}(\partial \Omega) $,
\begin{equation}\label{eq:DN}
\mathscr{B}f:= \frac{\partial}{\partial \nu}F\Big{|}_{\partial \Omega},
\end{equation}
where $F$ satisfies $\Delta F=0$ in $\Omega$ and $F=f$ on $\pa \om$, then the nonlinear boundary diffusion equation \eqref{1} can be rewritten as
\begin{equation}\label{eq:DtNdiffusion}
   \partial_{t}u^{p}=   -  \mathscr{B} u  -  a u     \quad  \mbox{on }     \partial \Omega\times (0,\infty).
\end{equation}
Therefore, the equation \eqref{1} can be viewed as a nonlinear evolution equation  with a nonlocal diffusion operator on the manifold $\partial\Omega$ of  principal symbol the same as the $1/2-$Laplace operator.  

If $\Omega=\R^n_+=\{x=(x',x_n): x_n>0\}$ and $a\equiv 0$, then the equation \eqref{1} becomes the fractional porous medium equation on the  Euclidean space 
\begin{equation}\label{eq:fpme1}
   \partial_{t}u^{p}=   -  (-\Delta)^{\sigma}u    \quad  \mbox{on }     \R^{n-1}\times (0,\infty)
\end{equation}
with $\sigma=1/2$. This equation \eqref{eq:fpme1} with $\sigma=1/2$ has been systematically studied by de Pablo-Quir\'os-Rodr\'iguez-V\'azquez \cite{PQRV}, which was later extended to the general fractional Laplace operator $(-\Delta)^{\sigma}$ for all $\sigma\in (0,1)$ by themselves in \cite{PQRV2}. V\'azquez \cite{Vaz}  established the existence, uniqueness and main properties of the fundamental solutions to \eqref{eq:fpme1}, and obtained some asymptotics properties. When $p$ is the critical Sobolev exponent, asymtotic behavior was studied by the first two authors in \cite{JX2014}. Bonforte-V\'azquez \cite{BV2014} obtained weighted global integral estimates for \eqref{eq:fpme1} that allow to establish existence of
solutions for classes of large data, and obtain quantitative pointwise lower estimates of the positivity of the solutions, depending only on the norm of the initial data in a certain ball. These estimates were later improved by V\'azquez-Volzone \cite{JL2013Optimal}. V\'azquez-de Pablo-Quir\'os-Rodr\'iguez \cite{VPQR} proved that weak solutions of \eqref{eq:fpme1} are classical for all positive times. More general nonlocal porous medium equations were studied in Bonforte-Endal \cite{BE22}. There are also studies on the fractional porous medium equation \eqref{eq:fpme1} posed in bounded domains with the Dirichlet condition:
\begin{equation}\label{eq:fpme2}
 \left\{
\begin{aligned}
   \partial_{t}u^{p}&=   -  (-\Delta)^{\sigma}u    \quad  \mbox{on }     \mathcal{V}\times (0,\infty),\\
   u&=  0    \quad  \mbox{on }    (\R^{n-1}\setminus \mathcal{V})\times (0,\infty),
\end{aligned}
\right.
\end{equation}
where $\mathcal{V}\subset \R^{n-1}$ is a bounded domain. The wellposedness and regularity of this equation, as well as their extensions to more general nonlocal porous medium equations, have been obtained by Bonforte-Sire-V\'azquez \cite{BSV2015},  Bonforte-V\'azquez \cite{BV2015}, Bonforte-Figalli-Ros-Oton \cite{BFR},  Bonforte-Figalli-V\'azquez \cite{BFV18}, Bonforte-Ibarrondo-Ispizua \cite{BII22}, Brasco-Volzone \cite{BV22} and Franzina-Volzone \cite{FV23}.  Note that $\mathcal{V}$ in \eqref{eq:fpme2} is a Euclidean open set in $\R^{n-1}$, but not a boundary of a domain $\Omega\subset\R^{n}$. Therefore, even though the equations \eqref{1} and \eqref{eq:fpme2} share some similarities in principles, they are two different equations. 

In this paper, we would like to study the large time behavior of the solutions to the nonlinear boundary  diffusion equation \eqref{1} with Sobolev subcritical exponents:
\begin{equation}\label{2}
\left\{
  \begin{aligned}
  & p\in (0,1)\cup (1,+\infty)  \quad  &\mbox{if }   n=2, \\
  &  p\in (0,1)\cup \left(1,  \frac{n}{n-2}\right) \quad   &\mbox{if }    n\geq3.
\end{aligned}
\right.
 \end{equation}
The dynamics of the solutions to \eqref{1}--\eqref{eq:initial} will depend on the range of $p$ and the sign of the first eigenvalue of the operator $\mathscr{B} +a $ defined by
\begin{equation}\label{eq:lambda1}
\lda_1= \inf_{u\in H^1(\om)}\left\{\displaystyle \int_{\om} |\nabla u|^2\,\ud x + \int_{\pa \om} au^2\,\ud S: \displaystyle \int_{\pa \om} u^2\,\ud S =1\right\},
\end{equation}
where $\,\ud S$ is the area element of $\partial\Omega$.
\begin{thm}\label{thm:wellposedness}
Let $n\ge 2$,  $ \Omega \subset \R^n$ be a bounded smooth domain,   $p$ satisfy \eqref{2} and $a\in C^\infty(\pa \om)$. Suppose that $u_0\in C^\infty(\pa \om)$ and $u_0$ is positive. Then  \eqref{1}--\eqref{eq:initial} admits a unique positive  smooth solution $u $ on $  \overline\Omega\times[0,T^*)$, where $[0,T^*)$, with $0<T^*\le \infty$, is the largest interval on which  $u$ is positive.  Moreover, there exists $C>0$ depending only on $n,p,a,\Omega, \max_{\partial\Omega}u_0$ and $\min_{\partial\Omega}u_0$ such that
\begin{itemize}
\item[(i)] if $\lambda_1(p-1)< 0$, then $T^*=\infty$ and
\[
\frac{1}{C}(t+1)^\frac{1}{p-1}\le u(x,t)\le C (t+1)^\frac{1}{p-1}\quad\mbox{on } \overline\Omega\times[0,\infty);
\]

\item[(ii)] if $\lambda_1=0$, then $T^*=\infty$ and 
\[
\frac{1}{C}\le u(x,t)\le C\quad\mbox{on } \overline\Omega\times[0,\infty);
\]

\item[(iii)] if $\lambda_1(p-1)>0$, then $T^*<\infty$ and 
\[
\frac{1}{C}(T^*-t)^{\frac{1}{p-1}}\le u(x,t)\le C(T^*-t)^{\frac{1}{p-1}}\quad\mbox{on } \overline\Omega\times[0,T^*).
\] 
\end{itemize}

\end{thm}

Part (iii) in Theorem \ref{thm:wellposedness} shows that  the solution will be extinct in finite time if $\lambda_1>0$ and $p>1$, and will blow up in finite time if $\lambda_1<0$ and $0<p<1$. The assumption of $p$ being Sobolev subcritical is used crucially in Propositions \ref{prop:existencesteady} and \ref{P7-1}, which will be used to prove Theorem \ref{thm:wellposedness}.

The  quantitative estimates in Theorem \ref{thm:wellposedness} can be illustrated via the separable solutions 
\[
U_c(x,t)=\varphi(x)b_c(t)
\] 
of \eqref{1}, where 
\begin{equation}\label{eq:seperatet}
b_c(t)= \left\{
  \begin{aligned}
(c+t)^{\frac{1}{p-1}} &\quad\mbox{if}\quad \lda_1(p-1)<0,   \\
c^{\frac{1}{p-1}}  &\quad\mbox{if}\quad \lda_1=0,   \\
(c-t)^{\frac{1}{p-1}}& \quad\mbox{if}\quad \lda_1(p-1)>0
\end{aligned}
\right.
\end{equation} 
is the solution of the ODE
\[
\partial_t [b(t)]^p= -\mathrm{sgn}(\lda_1)\frac{p}{|p-1|} b(t)
\] 
with the initial data $b(0)=c^{\frac{1}{p-1}}$ for some $c>0$, and $\varphi$ is a positive solution of 
\be \label{eq:steady}
-\Delta \varphi=0 \quad \mbox{in }\om, \quad \pa_\nu \varphi+a \varphi = \mathrm{sgn}(\lda_1)\frac{p}{|p-1|}\varphi^p \quad \mbox{on }\pa \om
\ee
with $ \mathrm{sgn}(\lda_1)\in \{-1,0,1\}$ denoting as the sign of $\lda_1$. The equation \eqref{eq:steady} is equivalent to
\begin{equation}\label{eq:stationary}
\mathscr{B} \varphi +a \varphi- \mathrm{sgn}(\lda_1)\frac{p}{|p-1|}\varphi^p=0 \quad \mbox{on }\pa \om.
\end{equation}
Since $p$ is Sobolev subcritical, the existence of $\varphi$ can be obtained by the variational method; see Proposition \ref{prop:existencesteady}. If $\lambda_1(p-1)<0$, then the solution of \eqref{eq:steady} is unique; see Proposition \ref{prop:unique}. Following the definition of Adams-Simon \cite{AS} (pages 229-230), we call a solution $\varphi$ of \eqref{eq:stationary} is integrable if for every nonzero $\phi\in \mathrm{Ker}\mathcal{L}_\varphi:=\{\phi\in H^{\frac12}(\pa\Omega): \mathcal{L}_\varphi\phi=0\}$, where
\begin{equation}\label{eq:linearized}
\mathcal{L}_\varphi:=\mathscr{B}+a- \mathrm{sgn}(\lda_1)\frac{p^2}{|p-1|}\varphi^{p-1},
\end{equation}
there exists a family $\{\varphi_s\}_{s\in(-1,1)}$ of solutions to \eqref{eq:stationary} such that $\varphi_s\to\varphi$ in $C^2(\partial\Omega)$ and $(\varphi_s-\varphi)/s\to\phi$ in $L^2(\partial\Omega)$ as $s\to 0$.

Furthermore, the above separable solutions characterize the limits of all the solutions in Theorem \ref{thm:wellposedness} as $t\to T^*$. To show the stability of these separable solutions, we shall use the  following changes of variables: 
\begin{equation} \label{eq:changeofvariable}
 \left\{
\begin{array}{lll}
w(x,\tau)&=u(x,t)/b_1(t) \ \mbox{ and }\  t=e^{\tau}-1,  \quad  &\mbox{if } \lda_1(p-1)<0;\\[1mm]
w(x,\tau)&=u(x,t)/b_1(t) \ \mbox{ and }\  t=\tau,   \quad &\mbox{if } \lda_1=0;\\[1mm]
w(x,\tau)&=u(x,t)/b_{T^*}(t) \ \mbox{ and }\  t=T^*(1-e^{-\tau}),   \quad &\mbox{if }\lda_1(p-1)>0. \\
\end{array}
\right.
\end{equation}
Then 
\begin{equation}\label{eq:scale}
 \left\{
\begin{aligned}
 \Delta w &= 0 \quad  \mbox{in }  \Omega\times(0, \infty ),        \\
\partial_{\tau} w^p&=-  \partial_{\nu}w  -  a w +    \mathrm{sgn}(\lda_1)\frac{p}{|p-1|}w^{p}   \quad  \mbox{on }  \partial \Omega \times(0, \infty ).
\end{aligned}
\right.
\end{equation}
It follows from Theorem \ref{thm:wellposedness} that $1/C\le w\le C$ on  $\overline\Omega\times[0,\infty)$ for  some $C\ge 1$ which depends only on $n,p,a,\Omega, \max_{\partial\Omega}u_0$ and $\min_{\partial\Omega}u_0$.

\begin{thm}\label{mainTH}
Let $u$ be the one in Theorem \ref{thm:wellposedness} and $w$ be defined as in \eqref{eq:changeofvariable}. Then $w(x,\tau)$ converges to a positive solution $\varphi$ of \eqref{eq:steady} in $C^{2}{(\overline\om)}$ as $\tau\to \infty$. Moreover, there exist  $C>0$, $\delta>0$ and $ \gamma>0$, all of which depend  only on $n,p,a,\Omega, \max_{\partial\Omega}u_0$ and $\min_{\partial\Omega}u_0$, such that 
\begin{itemize}
\item[(i)] if $\lda_1(p-1)< 0$, then 
 \begin{equation}\label{eq:convergencerate1}
\|w(\cdot,\tau)-\varphi\|_{C^{2}(\overline\Omega)} \leq  Ce^{- \tau} \quad \forall \  \tau>1;
\end{equation}
\item[(ii)] if $\lda_1= 0$, then 
 \begin{equation}\label{eq:convergencerate2}
\|w(\cdot,\tau)-\varphi\|_{C^{2}(\overline\Omega)} \leq  Ce^{-\gamma \tau} \quad \forall \  \tau>1;
\end{equation}
\item[(iii)] if $\lda_1(p-1)>0$, then either
 \begin{equation}\label{eq:convergencerate3}
C\tau^{-1}\le \| w(\cdot,\tau)-\varphi \|_{C^2(\overline\Omega)} \le C\tau^{-\delta}  \quad \forall \  \tau >1
\end{equation}
or
 \begin{equation}\label{eq:convergencerate4}
\| w(\cdot,\tau)-\varphi \|_{C^2(\overline\Omega)} \le Ce^{-\gamma \tau}    \quad \forall \  \tau > 1;
\end{equation}
\item[(iv)] if $\lda_1(p-1)>0$ and $\varphi|_{\partial\Omega}$ is integrable, then \eqref{eq:convergencerate4} holds.
\end{itemize}

\end{thm}

\begin{rem}
The decay rate $e^{- \tau}$ in \eqref{eq:convergencerate1} is sharp, since the separable solution $(2+t)^{\frac{1}{p-1}}\varphi(x)$ of \eqref{1} indeed satisfies this decay rate. The exponential exponent $\gamma$ in \eqref{eq:convergencerate2} and \eqref{eq:convergencerate4} is in fact equal to $\gamma_p$ which is defined in \eqref{sharprate}; see Theorem \ref{either fast or slower}. Moreover, this $\gamma_p$ is the same as in the linear case, and thus, no better rate shall be expected in \eqref{eq:convergencerate2} and \eqref{eq:convergencerate4} in this degree of generality. Furthermore, we obtain a higher order expansion \eqref{eq:error-exp-decay-rate-classify} of the solution in the exponential decay case (in fact, one can expand it to arbitrary orders).
\end{rem}

When $\lambda_1>0$,  then the diffusion operator $\mathscr{B} +a $ is coercive, and our results are the analogues of those for the classical nonlinear diffusion equation in bounded domains:
\begin{equation}\label{eq:fde}
 \left\{
\begin{aligned}
& \partial_{t}u^{q}  =\Delta u \quad                           \mbox{in }   \Omega \times (0,\infty),\\
&u              =  0  \quad   \mbox{on } \partial \Omega \times [0,\infty).
\end{aligned}
\right.
\end{equation}
In the case of $0<q<1$, the equation \eqref{eq:fde} is called the porous medium equation. It is a slow diffusion equation, which means that if $u_0$ is compactly supported in $\Omega$, then the solution $u(\cdot,t)$ with such initial data will still be compactly supported in $\Omega$ at least for a short time. Suppose 
$u_0\in C^1(\overline\Omega)$, $u_0\ge 0$ in $\overline\Omega$,  $u_0=0$ on  $\partial\Omega$, and $u_0\not\equiv 0$  in $\Omega$, then part (i) in Theorem \ref{thm:wellposedness} and the estimate in Theorem \ref{mainTH} under the Lipschitz norm (after a waiting time) were proved by Aronson-Peletier \cite{AP1981}. See V\'azquez \cite{Vaz2004} for another similar stability result of the separable solutions to the porous medium equation. In \cite{JRX}, the first two authors and X. Ros-Oton recently proved the optimal regularity of the solution $u$, improved the convergence from the Lipschitz topology to the $C^{2,q}(\overline\Omega)$ topology and obtained its finer asymptotics.  In the case of $1<q<+\infty$, the equation \eqref{eq:fde} is called the fast diffusion equation. It has infinite speed of propagation, which means that if the initial data is nonnegative and not identically zero, then the solution will be positive everywhere in $\Omega$ immediately. It is known from the work of Sabinina \cite{Sabinina1,Sabinina2} that solutions of \eqref{eq:fde} with positive initial data will be extinct after a finite time $T^*>0$. If $q$ is a Sobolev subcritical exponent, then under the assumption that $\partial_t u, \nabla u, \nabla\partial_t u, \nabla^2 u\in C(\overline \om\times (0,T^*))$,  Berryman-Holland \cite{Berryman} proved the integral estimates in Theorem \ref{thm:wellposedness}  for the solutions of \eqref{eq:fde}, as well as the convergence in Theorem \ref{mainTH} but just along a sequence of times in $H^1_0(\Omega)$. In 2000, Feireisl-Simondon \cite{Eduard2000Convergence}  proved the uniform convergence without the regularity assumption. Later, Bonforte-Grillo-V\'azquez \cite{Grillo0Behaviour} proved  the uniform convergence of the  relative error, and Bonforte-Figalli \cite{2019Sharp} quantified the convergence rate of the relative error and obtained a sharp exponential rate in generic domains.  See also Akagi \cite{2021Sharprate} for a different proof. Recently, the first two authors proved Berryman-Holland's regularity assumption and established the optimal regularity in \cite{2019Optimal, JX22}. As applications, they improved the  uniform convergence of the relative error to be in the $C^2$ topology in \cite{2019Optimal},  and  proved the polynomial decay rate for all smooth domains  in \cite{JX20}.  More  recently, Choi-McCann-Seis \cite{CMS} proved that the relative error either decays exponentially with the sharp rate,  or else decays algebraically at a rate $1/t$ or slower. Furthermore, they obtained higher order asymptotics, which refines and confirms a conjecture of Berryman-Holland \cite{Berryman}. Besides these results, there are many other works on the quantitative properties of the solutions to the fast diffusion equation \eqref{eq:fde} in bounded domains, and we refer to DiBenedetto-Kwong-Vespri \cite{DKV}, Galaktionov-King \cite{GKing}, Bonforte-V\'azquez \cite{BVjfa, BV}, Akagi \cite{Akagi}, Jin-Xiong \cite{JX20}, Sire-Wei-Zheng \cite{SWZ} and the references therein.

When $\lambda_1<0$, then one may look for  analogues between \eqref{1} and 
\begin{equation}\label{eq:fde2}
 \left\{
\begin{aligned}
& \partial_{t}u^{q}  =\Delta u + c u\quad                           \mbox{in }   \Omega \times (0,\infty),\\
&u              =  0  \quad   \mbox{on } \partial \Omega \times [0,\infty),
\end{aligned}
\right.
\end{equation}
where $c$ is a positive constant larger than the first Dirichlet eigenvalue of the Laplace operator $\Delta$ on $\Omega$. Both the solution of \eqref{1} with $0<p<1$ and the solution of \eqref{eq:fde2} with $0<q<1$ blow up in finite time, where the latter conclusion was proved by Galaktionov \cite{Galaktionov81}. However, there is a notable difference on the uniform lower bound in part (iii) of Theorem \ref{thm:wellposedness}. If $\Omega$ is large, and the support of the initial data $u_0$ is small and sufficiently away from $\pa\Omega$, then  Theorem 1 of Galaktionov \cite{Galaktionov} shows that the blow up for the solution of \eqref{eq:fde2} is localized. That is, the support of the solution of \eqref{eq:fde2} stays uniformly away from $\pa\Omega$ up to the blow up time, so that it is impossible to prove proper uniform lower bounds for them. However, the solution of \eqref{1} blows up everywhere uniformly. The main reason why  the uniform lower bound holds for  the solution of \eqref{1} with $0<p<1$ is that the solution satisfies a quantitative positivity estimate in Theorem \ref{thm:infinitespeed} in the below. This is different from \eqref{eq:fde2} which fails to be so for $0<q<1$. This positive estimate implies the solutions of \eqref{1} has infinite speed of propagation on $\partial\Omega$, which is analogous to the phenomenon of infinite speed of propagation observed for porous medium equations with nonlocal operators in Euclidean space by de Pablo-Quir\'os-Rodr\'iguez-V\'azquez \cite{PQRV} and  Bonforte-Figalli-Ros-Oton \cite{BFR}. 

\begin{thm}\label{thm:infinitespeed}
Let $n\ge 2$, $ \Omega \subset \R^n$ be a bounded smooth domain, $p$ satisfy \eqref{2}, $a\in C^\infty(\pa \om)$, $u_0 \in C^\infty(\pa \om) $ be a  positive function, and $u$ be a smooth positive solution of  \eqref{1}--\eqref{eq:initial} on $\overline\Omega\times[0,T]$ for some $T>0$.  Then there exist $t_0>0$ and $\va_0>0$, both of which depend only on $n,p,a,\Omega$, $\max_{\partial\Omega} u_0$ and $\|u_0\|_{L^p(\partial\Omega)}$, but \emph{not} on $\min_{\partial\Omega} u_0$, such that
\[
u(x,t)\ge \va_0 t^{1/p}\quad\mbox{for all }(x,t)\in\overline\Omega\times[0,t_0].
\]  
\end{thm}

Since the $\va_0$ in Theorem \ref{thm:infinitespeed} does not depend on $\min_{\partial\Omega} u_0$, then one can consider the equation \eqref{1} with nonnegative initial data by approximations, and its solutions will be immediately positive everywhere. Finite time blow up for the porous medium equations with reactions of type \eqref{eq:fde2} has  been very extensively studied in the literature, and we refer the readers to the survey Galaktionov-V\'azquez \cite{GVsurvey} for more references.

The estimates in Theorem \ref{thm:wellposedness} are sometimes called the global Harnack inequality in the literature  (see, for example, Bonforte-V\'azquez \cite{BV2015} for the fractional porous medium equations in Euclidean spaces,  Bonforte-Vazquez \cite{BVjfa} for a result of having a lower and upper bound in terms of Barenblatt profiles for \eqref{eq:fde} in the global case that $\Omega=\R^n$, and also Bonforte-Simonov \cite{BS20} for a complete characterization of the maximal set of initial data that produces solutions which are pointwisely trapped between two Barenblatt solutions). Parts (i) and (ii) of Theorem \ref{thm:wellposedness}  follow directly from the comparison principle and a short time existence theorem. Part (iii) is more delicate. From the comparison principle, we first show that $T^*<+\infty$, from which we know that $u$ either blows up or becomes extinct at the time $T^*$ at least at one point on $\pa\Omega$. Secondly, we show elliptic type weak Harnack inequalities and local maximum principles for solutions of \eqref{1} on each time slice and then use them to show that $\|u\|_{L^{p+1}(\pa\Omega)}$ either blows up or becomes extinct at the time $T^*$.  Harnack type inequalities for solutions of classical fast diffusion equations \eqref{eq:fde} were obtained in Bonforte-Simonov \cite{BS2020Adv} and Bonforte-Dolbeault-Nazaret-Simonov \cite{BDNS2021}. In this step, the elliptic type weak Harnack inequalities and local maximum principles  for \eqref{1} are obtained by considering an evolution equation of a curvature-like quantity, which was used before by the first two authors in \cite{2019Optimal} for \eqref{eq:fde}. Then by using a similar argument to that in Berryman-Holand \cite{Berryman}, we prove that $\|u\|_{L^{p+1}(\pa\Omega)}$ is bounded from below and above by two uniform multiples of $(T^*-t)^{\frac{1}{p-1}}$. Finally, the uniform upper bound of $u$ is obtained by Moser's iteration, and the uniform lower bound of $u$ is obtained by a quantitative Hopf's lemma (Lemma \ref{lem:normal derivative lower bound}). The proof of the uniform lower bound is the place where one can sense a nonlocal nature of the nonlinear boundary diffusion problem \eqref{1}. 

 Part (i) of Theorem \ref{mainTH} follows directly from the comparison principle again.  The convergence, and the rates in part (ii) and part (iv) of Theorem \ref{mainTH}, as well as the upper bound in \eqref{eq:convergencerate3}, are obtained by using {\L}ojasiewicz's inequality in infinite dimensional spaces developed by Simon \cite{Simon1983Asymptotics}. Whether or not the exponent $\theta$ in the {\L}ojasiewicz inequality \eqref{40} can reach $\frac{1}{2}$ would lead to either exponential convergence rates or algebraic convergence rates. The integrability condition on the solution of \eqref{eq:stationary} introduced by Adams-Simon \cite{AS} (see also Allard-Almgren \cite{AA} for an earlier integrability hypothesis) implies the {\L}ojasiewicz inequality with exponent $\frac12$, leading to the exponential convergence. If $\lambda_1(p-1)<0$, then the smooth positive solution $\varphi$ of \eqref{eq:steady}  is unique, and the linearized operator at $\varphi$ has a trivial kernel; see  Proposition \ref{prop:nondegenerate}. Thus, $\varphi$ is integrable.  If $\lambda_1=0$, then $\mathrm{Ker}\,\mathcal{L}_\varphi=\mbox{span}\{\varphi\}$, and thus, $\varphi$ is also integrable, by choosing $\varphi_s=(1+s)\varphi$. The dichotomy between \eqref{eq:convergencerate4} and the lower bound in \eqref{eq:convergencerate3} is proved by adapting  the arguments of  Choi-McCann-Seis \cite{CMS} for the classical fast diffusion equation \eqref{eq:fde}.

This paper is organized as follows. In Section \ref{sec:Preliminaries}, we collect some elementary inequalities and prove several elliptic results as a preparation. In Section \ref{sec:existence}, we show the short time existence of the solutions and its infinite speed of propagation. In Section \ref{sec:extinction}, we show the finite time extinction or finite time blow up phenomena, and prove the uniform lower and upper integral bounds for the rescaled solution $w$ defined in \eqref{eq:changeofvariable}. In Section \ref{sec:uniformbound}, we prove the uniform bounds in Theorem \ref{thm:wellposedness}. In Section \ref{sec:convergence}, we prove the convergence results in Theorem \ref{mainTH}, including the sharp convergence rate, and show higher order asymptotics. 

\bigskip

\noindent \textbf{Acknowledgement:} This is an improvement of part of the third author’s thesis at Beijing Normal University. We would like to thank the anonymous referee for his/her careful reading of the paper and for invaluable suggestions that greatly improved the presentation of the paper.

\section{Steady states and some quantitative elliptic estimates}\label{sec:Preliminaries}

Let $ \Omega \subset \R^n$, $n \geq 2 $,  be a bounded smooth domain. We first recall a Sobolev inequality and a trace inequality, which will make it easier for us to refer later. For $0<q_1<+\infty$ if $n=2$, and $0 < q_1 \le   \frac{n+2}{n-2} $ if $n\ge 3$, there exists $C_1>0$ depending only on $n,q_1$ and $\Omega$ such that
\begin{equation}\label{eq:sobolev00}
\left(\int_\Omega |u|^{q_1+1}\,\ud x\right)^\frac{2}{q_1+1}\le  C_1\left(\int_{\om} |\nabla u|^2\,\ud x + \int_{\pa \om} u^2\,\ud S\right)\quad\mbox{for all }u\in H^1(\Omega).
\end{equation}
For $0<q_2<+\infty$ if $n=2$, and $0 < q_2 \le   \frac{n}{n-2} $ if $n\ge 3$, there exists $C_2>0$ depending only on $n,q_2$ and $\Omega$ such that
\begin{equation}\label{thm:trace}
\left(\int_{\partial\Omega} |u|^{q_2+1}\,\ud S\right)^\frac{2}{q_2+1}\le  C\left(\int_{\om} |\nabla u|^2\,\ud x + \int_{\pa \om} u^2\,\ud S\right)
\end{equation}
Consequently, by a compactness argument, we  have the following coercivity.  
\begin{prop}\label{prop:sobolev}
Let $a\in C^\infty(\pa \om)$. Assume $\lambda_1>0$, where $\lambda_1$ is defined in \eqref{eq:lambda1}. Then there exists $C>0$ depending only on $n,\Omega,\lambda_1$ and $a$ such that
\[
\int_{\om}( |\nabla u|^2+u^2)\,\ud x \le C\left(\int_{\om} |\nabla u|^2\,\ud x + \int_{\pa \om} au^2\,\ud S\right) \quad\mbox{for all }u\in H^1(\Omega).
\]
\end{prop}

We will use the following known local maximum principle and weak Harnack inequality. For reader's convenience, we sketch their proofs.

\begin{lem}[Local maximum principle]\label{lem:localmax}
Suppose $w$ is a positive smooth (in $\overline\Omega$) solution of
\begin{equation*}
  \left\{
    \begin{array}{ll}
        -\Delta w  \le 0          \quad  \mbox{in }   \Omega \\
           \frac{\partial}{\partial \nu} w  \leq  gw      \quad \mbox{on }   \partial \Omega
         \end{array}
  \right.
\end{equation*}
with $g^{+}\in L^\infty(\partial \Omega)$, where $g^+(x)=\max(g(x),0)$. Then there exist $C>0$, which depends only on $n, \Omega$ and $\|g^{+}\|_{L^\infty(\partial \Omega)}$, such that
$$
\sup_{\partial\Omega}w \leq C\| w  \|_{L^{1}(\partial\Omega)}.
$$
\end{lem}
\begin{proof}
For $k\ge 1$, multiplying $w^{k}$ on both sides and integrating by parts, we obtain
\[
k\int_{\Omega} w^{k-1} |\nabla w|^2\,\ud x\le  \int_{\partial\Omega}w^{k}\frac{\partial}{\partial \nu} w \,\ud S\le  \int_{\partial\Omega} gw^{k+1}\,\ud S.
\]
If we let $v=w^{\frac{k+1}{2}}$, then this implies that
\[
\int_{\Omega} |\nabla v|^2\,\ud x\le \frac{C(k+1)^2}{4k} \int_{\partial\Omega}v^2 \,\ud S.
\]
By the trace inequality \eqref{thm:trace},
\[
\left(\int_{\partial\Omega} |v|^q\,\ud x\right)^\frac{2}{q}\le \frac{C(k+1)^2}{4k} \int_{\partial\Omega}v^2 \,\ud S,
\]
where $q=\frac{2(n-1)}{n-2}$ if $n\ge 3$, and $q=4$ if $n=2$. 
Then it follows from the standard Moser's iteration (see, e.g., Chapter 8 of Gilbarg-Trudinger \cite{1983Trudinger}) that 
$$
\sup_{\partial\Omega}w \leq C\| w  \|_{L^{2}(\partial\Omega)}.
$$
Then the conclusion follows from the fact that $\| w  \|_{L^{2}(\partial\Omega)}\le (\sup_{\partial\Omega}w)^{\frac12} (\| w  \|_{L^{1}(\partial\Omega)})^{\frac12}$.
\end{proof}

We need an  auxiliary lemma for the weak Harnack inequality in a form that we need.

\begin{lem}\label{lem:inversecontrol}
Let $u\in C^2(\overline\Omega)$ be a positive function such that $-\Delta u\ge 0$ in $\Omega$. Then for every $0<\delta\le 1$, there exists $C>0$ depending only on $n,\Omega$ and $\delta$ such that
\[
\|u\|_{L^\delta(\Omega)}\ge C \|u\|_{L^1(\partial\Omega)}.
\]
\end{lem}
\begin{proof}
Let $\varphi(\xi)=u(\xi)$ for every $\xi\in\partial\Omega$. Let $P(x,\xi):\Omega\times\partial\Omega\to\R^+$ be the Poisson kernel. Note that we have the following estimates for the Poisson kernel (see, e.g., Theorem 1 in Krantz \cite{Krantz}):
\begin{equation}\label{eq:estpoisson}
c_1\frac{dist(x,\partial\Omega)}{|x-\xi|^n}\le P(x,\xi)\le c_2\frac{dist(x,\partial\Omega)}{|x-\xi|^n},
\end{equation}
where $c_1$ and $c_2$ are two positive constants depending only on $n$ and $\Omega$. Let
\[
 U(x):=\int_{\partial\Omega}P(x,\xi)\varphi(\xi)\,\ud S(\xi).
\]
Then
\begin{align*}
\int_{\Omega}U(x)^\delta\,\ud x&=\int_{\Omega}U(x)^{\delta-1}\left(\int_{\partial\Omega}P(x,\xi)\varphi(\xi)\,\ud S(\xi)\right)\,\ud x\\
&= \int_{\partial\Omega}\varphi(\xi) \left(\int_{\Omega}U(x)^{\delta-1}P(x,\xi)\,\ud x\right)\,\ud S(\xi).
\end{align*}
By the reverse H\"older inequality, we have
\begin{align*}
\int_{\Omega}U(x)^{\delta-1}P(x,\xi)\,\ud x&\ge \left(\int_{\Omega}U(x)^\delta\,\ud x \right)^{\frac{\delta-1}{\delta}} \left(\int_{\Omega}P(x,\xi)^\delta\,\ud x\right)^{\frac{1}{\delta}}.
\end{align*}
Since $\overline\Omega$ is smooth and compact, it satisfies a uniform interior ball condition, that is, there exists $r>0$ such that for every $\xi\in\partial\Omega$, there exists $x_\xi\in\Omega$ such that the ball $B_r(x_\xi)\subset\Omega$, and $\overline B_r(x_\xi) \cap\partial\Omega=\{\xi\}$. Hence,
\[
\int_{\Omega}P(x,\xi)^\delta\,\ud x \ge c_1 \int_{B_{r/2}(x_\xi)}\left(\frac{dist(x, \partial \Omega)}{|x-\xi|^n}\right)^\delta\,\ud x\ge c_1 |B_1|\left(\frac{1}{2^{n+1}r^{n-1}}\right)^\delta \left(\frac{r}{2}\right)^{n}>0.
\]
Therefore,
\[
\left(\int_{\Omega}U(x)^\delta\,\ud x\right)^{\frac{1}{\delta}} \ge C \int_{\partial\Omega}\varphi(\xi) \,\ud S(\xi).
\]
Since $u\ge U$ in $\Omega$ by the maximum principle, the conclusion follows.
\end{proof}

\begin{lem}[Weak Harnack inequality]\label{lem:weakharnack}
Suppose $w$ is a positive smooth (in $\overline\Omega$) solution of
\begin{equation*}
  \left\{
    \begin{array}{ll}
        -\Delta w  \ge 0          \quad  \mbox{in }   \Omega \\
           \frac{\partial}{\partial \nu} w  \geq  gw      \quad \mbox{on }   \partial \Omega
         \end{array}
  \right.
\end{equation*}
with $g^{-}\in L^\infty(\partial \Omega)$, where $g^-(x)=-\min(g(x),0)$. Then there exists $C>0$, which depends only on $n, \Omega$ and $\|g^{-}\|_{L^\infty(\partial \Omega)}$, such that
$$
\inf_{\partial\Omega}w \geq \frac{1}{C}\| w  \|_{L^{1}(\partial\Omega)}.
$$
\end{lem}
\begin{proof}
By the proof of Lemma A.1 of Han-Li \cite{HanLi}, there exist $C>0$ and $\delta\in (0,1)$, which depend only on $n, \Omega$ and $\|g^{-}\|_{L^\infty(\partial \Omega)}$, such that
\[
\inf_{\partial\Omega}w \geq \frac{1}{C}\| w  \|_{L^{\delta}(\Omega)}.
\]
Then the conclusion follows from Lemma \ref{lem:inversecontrol}. 
\end{proof}

We will also use the following quantitative Hopf's lemma for nonnegative harmonic functions  when proving the uniform lower bound in part (iii) of Theorem \ref{thm:wellposedness}. It should be known in the literature, but we cannot find a reference. Hence, we provide a proof.

\begin{lem}[A quantitative Hopf's lemma]\label{lem:normal derivative lower bound}
Let $u\in C^2(\overline\Omega)$ be a nonnegative function such that $\Delta u= 0$ in $\Omega$. If $u(x_0)= 0 $ for some $x_0 \in \pa \Omega$, then there exists $C>0$ depending only on $n$ and $\Omega$ such that
\[
-\pa_{\nu} u (x_0) \ge C \int_{\pa \Omega}u \,\ud S.
\]
\end{lem}
\begin{proof}
Since $u$ is nonnegative, if $\int_{\pa \Omega}u \,\ud S=0$, then by the strong maximum principle, $u\equiv0$ in $\overline\Omega$, by which the lemma clearly holds. 

If $\int_{\pa \Omega}u \,\ud S>0$, by scaling, we can assume that $\int_{\pa \Omega}u \,\ud S=1$. Similar to the proof of Lemma \ref{lem:inversecontrol}, we denote $\varphi(\xi)=u(\xi)$ for every $\xi\in\partial\Omega$. Let $P(x,\xi):\Omega\times\partial\Omega\to\R^+$ be the Poisson kernel. Then we have representation
\[
 u(x):=\int_{\partial\Omega}P(x,\xi)\varphi(\xi)\,\ud S(\xi) \quad x\in \Omega.
\]
Without loss of generality, we assume that $0\in\Omega$ and $x_0=(0,\cdots,0,-1)$. Denote $ r=\frac{1}{2}dist(0, \pa \Omega)$. Then we know from \eqref{eq:estpoisson} that 
\[
u(x) \ge \tilde{C}(\Omega,n)>0 \quad  \forall\,x \in \overline{B}_r.
\]
Let
$$
U:=\Omega \cap \{ x  \ | \    x_{n} < 0   \}, \quad B_r':=\left\{x: |x|<r, x_n = 0 \right\},
$$
$\eta\in C^\infty_c(B_r')$ be nonnegative everywhere and equals to $1$ in $B_{r/2}'$, and
$\psi$ be the solution of
\begin{align*}
\Delta \psi &= 0 \mbox{ in }  U,\\
\psi &= \tilde{C} \eta \mbox{ on } \ \left\{x: |x|<r, x_n = 0 \right\}, \\
   \psi &= 0 \mbox{ on } \ \pa U \setminus \left\{x: |x|<r, x_n = 0 \right\}. 
\end{align*}
Then by the maximum principle, it follows that
$$
u \ge \psi>0  \mbox{ in }  U.
$$
Since $u(x_0)=\psi(x_0)=0$, we have
\[
\pa_\nu (u-\psi)(x_0) \le 0.
\]
Finally it gives 
\[
-\pa_\nu u(x_0) \ge -\pa_\nu \psi (x_0) >0  
\]
where we used Hopf's Lemma for harmonic functions in the last inequality. The lemma is proved by choosing $C=-\pa_\nu \psi (x_0)$ and by recalling the normalization $\int_{\pa \Omega}u \,\ud S=1$ at the beginning of the proof.
\end{proof}

Next, we prove the existence of solutions of \eqref{eq:steady} or \eqref{eq:stationary}.
\begin{prop}\label{prop:existencesteady}
Let $a\in C^\infty(\pa \om)$ and $p$ satisfy \eqref{2}. Then there exists a positive smooth solution of \eqref{eq:steady}.
\end{prop}

\begin{proof}
For $u\in H^1(\om)$, $u\not\equiv 0$ on $\pa \om$, we define
\[
E_p[u]:= \frac{ \int_{\om} |\nabla u|^2\,\ud x + \int_{\pa \om} au^2\,\ud S}{  \left(\int_{\pa \om} |u|^{p+1}\,\ud S \right)^{\frac{2}{p+1}}}.
\]
Let
\begin{equation}\label{eq:energyY}
Y_p=\inf_{\substack{u\in H^1(\om), \\ u\not\equiv 0 \text{ on }\pa \om}} E_p[u].
\end{equation}
If $\lda_1> 0$, then by Proposition \ref{prop:sobolev}, \eqref{thm:trace} and H\"older's inequality, we have $Y_p>0$.  If $\lda_1<0$, then by noticing that  $E_p[\phi_1]<0$, where $\phi_1$ is an eigenfunction associated to $\lda_1$, we see that $Y_p<0$. If $\lda_1=0$, then on one hand, we have $E_p[u]\ge 0$ for all $u\in H^1(\om)$ so that $Y_p\ge 0$, and on the other hand, $E(\phi_1)\le 0$ so that $Y_p\le 0$. Therefore, if $\lda_1=0$, then $Y_p=0$. Hence, 
\[
\mathrm{sgn}(Y_p)=\mathrm{sgn}(\lda_1).
\]
Furthermore,  by an interpolation inequality (when $0<p<1$) or H\"older's inequality  (when $p>1$), we have 
\[
\int_{\pa \om} |a| u^2\,\ud S \le  \int_{\om} |\nabla u|^2\,\ud x + C_0 \left(\int_{\partial\Omega} |u|^{p+1}\,\ud S\right)^\frac{2}{p+1} \quad\mbox{for all } u\in H^1(\Omega),
\]
where $C_0>0$ depends only on $n,\om,p$ and $\|a\|_{L^\infty(\om)}$.  It follows that 
\[
\int_{\om} |\nabla u|^2\,\ud x + \int_{\pa \om} au^2\,\ud S \ge - C_0  \left(\int_{\partial\Omega} |u|^{p+1}\,\ud S\right)^\frac{2}{p+1},
\]
and thus,
\begin{equation}\label{eq:finiteyp}
Y_p\ge -C_0. 
\end{equation}
Since $p$ is subcritical, then by the standard variational method, the compact Sobolev embedding, and the fact that $E[u]=E[|u|]$, we have that $Y_p$ is achieved by some nonnegative function $\varphi\in H^1(\om)$ satisfying $\int_{\pa \om} \varphi^{p+1}\,\ud S=1$ and 
\[
-\Delta \varphi=0 \quad \mbox{in }\om, \quad \pa_\nu \varphi+a \varphi = Y_p \varphi^p \quad \mbox{on }\pa \om
\]
in the distribution sense.  By the regularity result of Cherrier \cite{Cherrier} and those for harmonic functions, we have $\varphi\in C^\infty(\Omega)\cap C^{1,\alpha}(\overline\Omega)$ for some $\alpha>0$.  Then by Hopf's Lemma, we see that $\varphi$ is positive in $\overline\Omega$, and then is smooth in $\overline\Omega$. Finally, since  $\mathrm{sgn}(Y_p)=\mathrm{sgn}(\lda_1)$, we know that if $\lda_1=0$, then $\varphi$ is a solution of  \eqref{eq:steady}, and if $\lda_1\neq 0$, then $\left(\frac{p}{|(p-1)Y_p|}\right)^{\frac{1}{1-p}} \varphi$  is a solution of  \eqref{eq:steady}. 
\end{proof}

The next two propositions are on uniqueness and non-degeneracy of the solutions to the stationary equation \eqref{eq:steady} for $\lambda_1(p-1)<0$.

\begin{prop}\label{prop:unique}
Let $a\in C^\infty(\pa \om)$ and $p$ satisfy \eqref{2}. If $\lambda_1(p-1)<0$, then there exists a unique positive smooth (in $\overline\Omega$) solution of
\[
-\Delta \varphi=0 \quad \mbox{in }\om, \quad \pa_\nu \varphi+a \varphi = \mathrm{sgn}(\lda_1)\frac{p}{|p-1|}\varphi^p \quad \mbox{on }\pa \om.
\]
\end{prop}
\begin{proof}
We only need to prove the uniqueness here.  

Suppose $u, v\in C^2(\overline\Omega)$ are two positive solutions. Suppose by contradiction that there exists $x_0\in\partial\Omega$ such that $u(x_0)<v(x_0)$.  

For $\lambda\ge 0$, define
\[
u_\lambda=\lambda u
\]
and 
\[
\bar\lambda=\inf\{\lambda\ge 0: u_\lambda\ge v\ \mbox{on }\partial\Omega\}.
\]
Since $u(x_0)<v(x_0)$, we have $\bar\lambda\ge 1$, and thus,
\[
\pa_\nu u_{\bar\lambda}+a u_{\bar\lambda} = \mathrm{sgn}(\lda_1)\frac{p}{|p-1|}{\bar\lambda}^{1-p}u_{\bar\lambda}^p\ge \mathrm{sgn}(\lda_1)\frac{p}{|p-1|}u_{\bar\lambda}^p \quad \mbox{on }\pa \om.
\]
Then, we have $-\Delta (u_{\bar\lambda}-v)=0 \  \mbox{in }\om$, and
\[
\pa_\nu (u_{\bar\lambda}-v)+a (u_{\bar\lambda}-v) + \mathrm{sgn}(\lda_1)\frac{p}{|p-1|}(u_{\bar\lambda}^p-v^p) \ge 0 \quad \mbox{on }\pa \om.
\]
By the definition of $\bar\lambda$, $u_{\bar\lambda}\ge v$ on $\partial\Omega$, and there exists $\bar x\in\partial\Omega$ such that $u_{\bar\lambda}(\bar x)=v(\bar x)$. It follows from the maximum principle and  Hopf's lemma for harmonic functions that
\[
u_{\bar\lambda}\equiv v\quad \mbox{ in }\overline\Omega.
\]
Hence, by the equations of $v$ and $u_{\bar\lambda}$, we have
\[
\mathrm{sgn}(\lda_1)\frac{p}{|p-1|}v^p=\pa_\nu v+a v = \pa_\nu u_{\bar\lambda}+a u_{\bar\lambda} = \mathrm{sgn}(\lda_1)\frac{p}{|p-1|}\bar\lambda^{1-p}u_{\bar\lambda}^p.
\]
Hence, $\bar\lambda=1$, and thus,
\[
u\equiv v.
\]
This is a contradiction to the existence of $x_0$. This finishes the proof.
\end{proof}

\begin{prop}\label{prop:nondegenerate}
Let $a\in C^\infty(\pa \om)$ and $p$ satisfy \eqref{2}. Suppose $\lambda_1(p-1)<0$, where $\lambda_1$ is defined in \eqref{eq:lambda1}. Let $\varphi$ be the solution in Proposition \ref{prop:unique}. Then the linearized operator $\mathcal{L}_\varphi$ defined in \eqref{eq:linearized} has a trivial kernel.
\end{prop}
\begin{proof}
Consider the eigenvalue problem:
\[
\mathcal{L}_\varphi\phi:=\mathscr{B}\phi+a\phi - \mathrm{sgn}(\lda_1)\frac{p^2}{|p-1|} \varphi^{p-1}\phi=\mu \varphi^{p-1} \phi.
\]
Since $\varphi$ is positive and satisfies
\[
\mathcal{L}_\varphi \varphi=p \varphi^{p-1} \varphi,
\]
then $p$ must be the first eigenvalue of $\mathcal{L}_\varphi$ and $\varphi$ must be a corresponding first eigenfunction. Therefore, all the eigenvalues of $\mathcal{L}_\varphi$ are larger than or equal to $p$. The conclusion follows.
\end{proof}

\section{Existence, uniqueness, and infinite speed of propagation}\label{sec:existence}

In this section, we first show that  the nonlinear boundary diffusion problem \eqref{1}--\eqref{eq:initial} has a unique solution on a small time interval. The following a priori estimates in Sobolev spaces were proved in Lemma 3.4 of Brendle \cite{brendle2002generalization}.

\begin{lem}\label{L2}
Let $0 < p < \infty$. Let $\phi$ be a smooth solution of the linear initial boundary value problem
\begin{equation*}
\begin{split}
  \Delta \phi   &=   0  \quad  \mbox{in }     \Omega\times (0,\infty),     \\
      \partial_t\phi &= -\frac{1}{p}u^{-(p-1)}\partial_\nu \phi + f  \quad  \mbox{on }     \partial \Omega\times (0,\infty),                 \\
    \phi &=   0       \quad  \mbox{on }     \partial \Omega\times \{ t=0 \},
\end{split}
\end{equation*}
where $c_0\le u\le C_0$ on $\Omega\times (0,\infty)$ for some positive constants $c_0$ and $C_0$.  In addition, we assume that $u$ satisfies
$$\|u\|_{W^{m,2}( \partial\Omega \times [0,T]) }   \leq    C_1$$ for some nonnegative integer $m$. Then there exists $C>0$ depending only on $n$, $p$, $\Omega$, $m$, $T$, $c_0$, $C_0$ and $C_1$ such that
$$
\|\phi\|_{W^{m+1,2}( \partial\Omega \times [0,T]) }    \leq    C \|f\|_{W^{m,2}(\partial\Omega\times [0,T])}.
$$
\end{lem}

Then, the short time existence for \eqref{1}--\eqref{eq:initial} follows from Lemma \ref{L2} and the  implicit function theorem. The proof is standard and we omit it here.
\begin{thm}\label{thm:short}
Let $n\ge 2$, $ \Omega \subset \R^n$   be a bounded smooth domain, $a\in C^\infty(\pa \om)$, $0<p<\infty$, and $u_0 \in C^\infty(\pa \om) $ be a  positive function. Then the initial boundary value problem \eqref{1}--\eqref{eq:initial} has a unique smooth positive solution on a small time interval.
\end{thm}

We will use the next comparison principle to show various dynamics of the solutions to \eqref{1}--\eqref{eq:initial}.

\begin{prop}[Comparison Principle]\label{prop:comparison}
Let $0< p<\infty$, $a\in C^\infty(\pa \om)$, and $b_j\in C^\infty(\overline \om)$ for $j=1,\cdots, n$. Let $c\in C^\infty(\pa \om)$ be positive everywhere. Suppose $u_1$ and $u_2$ are two smooth positive functions in $\overline\Omega\times[0,T)$ satisfying $u_1(\cdot,0)\le u_2(\cdot,0)$ in $\overline \Omega$, 
\[
\begin{split}
  \Delta u_1  - \sum_{j=1}^n b_j \partial_{x_j} u_1&\ge  0  \quad  \mbox{in }     \Omega\times (0,T),     \\
   \partial_{t}u_1^{p}&\le   - c \partial_{\nu}u_1  -  a u_1     \quad  \mbox{on }     \partial \Omega\times (0,T),
\end{split}
\]
and
\[
\begin{split}
  \Delta u_2 - \sum_{j=1}^n b_j \partial_{x_j} u_2 &\le   0  \quad  \mbox{in }     \Omega\times (0,T),     \\
   \partial_{t}u_2^{p}&\ge   - c \partial_{\nu}u_2  -  a u_2     \quad  \mbox{on }     \partial \Omega\times (0,T).
\end{split}
\]
Then
\[
u_1\le u_2 \quad\mbox{on }\overline\Omega\times[0,T).
\]
\end{prop}
\begin{proof}
The difference $u_1-u_2$ satisfies
\[
\begin{split}
&  -\Delta (u_1-u_2) +\sum_{j=1}^n b_j \partial_{x_j} (u_1-u_2) \le  0  \quad  \mbox{in }     \Omega\times (0,T),     \\
&   pu_1^{p-1}\partial_{t}(u_1-u_2)\le   - c \partial_{\nu}(u_1-u_2)  -g(x,t) (u_1-u_2)     \quad  \mbox{on }     \partial \Omega\times (0,T),
\end{split}
\]
where
\[
g(x,t)=a(x)+p(p-1)\partial_t u_2(x,t) \cdot \int_0^1 [\lambda u_1(x,t)+(1-\lambda) u_2(x,t)]^{p-2}\,\ud\lambda.
\]
Let $s\in (0,T)$. Then $u_1$ and $u_2$ are smooth and positive in $\overline\Omega\times[0,s]$, and $g$ is bounded on $\partial\Omega\times[0,s]$. Choose $C>0$ such that
\[
Cpu_1^{p-1}+g>0\quad\mbox{on }\partial\Omega\times[0,s].
\]
Let $v(x,t)=e^{-Ct}(u_1(x,t)-u_2(x,t))$. Then
\[
\begin{split}
  -\Delta v +\sum_{j=1}^n b_j \partial_{x_j} v&\le  0  \quad  \mbox{in }     \Omega\times (0,s],     \\
   pu_1^{p-1}\partial_{t}v&\le   -  c\partial_{\nu}v  -(Cpu_1^{p-1}+g) v    \quad  \mbox{on }     \partial \Omega\times (0,s],\\
   v(\cdot,0)& \le 0 \quad  \mbox{in }     \Omega.
\end{split}
\]
Then, it follows from the maximum principle and Hopf's lemma for elliptic equations that 
\[
v\le 0\quad  \mbox{on }     \overline \Omega\times [0,s].
\] 
That is,
\[
u_1\le u_2 \quad  \mbox{on }     \overline \Omega\times [0,s].
\]
Since $s\in (0,T)$ is arbitrary, we have
\[
u_1\le u_2 \quad  \mbox{on }     \overline \Omega\times [0,T).
\]
\end{proof}

Now, let us go back to the equations \eqref{1}--\eqref{eq:initial} for $p$ satisfying \eqref{2}. Using the comparison principle in Proposition \ref{prop:comparison}, we will derive the following estimates.

\begin{prop}\label{cor:boundednessbefore}
Let $n\ge 2$, $ \Omega \subset \R^n$   be a bounded smooth domain, $a\in C^\infty(\pa \om)$, $p$ satisfy \eqref{2}, $u_0 \in C^\infty(\pa \om) $ be a  positive function, and $u$ be a smooth positive solution of  \eqref{1}--\eqref{eq:initial} on $\overline\Omega\times[0,T]$ for some $T>0$. Then there exist $s_1>0$ depending only on $n,p,a,\Omega$ and $\min_{\partial\Omega} u_0$, $s_2>0$ depending only on $n,p,a,\Omega$ and $\max_{\partial\Omega} u_0$,  and $C>0$ depending only on $n,p,a$ and $\Omega$, such that
\begin{itemize}
\item[(i).] If $\lambda_1 (p-1) <0$, then
\[
\left(t+s_1\right)^\frac{1}{p-1}\varphi(x)\le u(x,t)\le (t+s_2)^\frac{1}{p-1}\varphi(x)\quad\mbox{on }\overline\Omega\times[0,T],
\]
where $\varphi$ is the unique positive solution of \eqref{eq:steady}.
\item[(ii).] If $\lambda_1 =0$, then
\[
\frac{s_1}{C}\le u(x,t)\le C s_2\quad\mbox{on }\overline\Omega\times[0,T].
\]
\item[(iii).] If $\lambda_1 >0$ and $p>1$, then
\begin{align*}
u(x,t)&\ge \frac{1}{C}(s_1-t)^\frac{1}{p-1}\quad\mbox{on }\overline\Omega\times[0,\min(T,s_1)],\\
u(x,t)&\le C(s_2-t)^\frac{1}{p-1}\quad\mbox{on }\overline\Omega\times[0,T].
\end{align*}
\item[(iv).] If $\lambda_1 <0$ and $0<p<1$, then
\begin{align*}
u(x,t)&\ge \frac{1}{C}(s_1-t)^\frac{1}{p-1}\quad\mbox{on }\overline\Omega\times[0,T],\\
u(x,t)&\le C(s_2-t)^\frac{1}{p-1}\quad\mbox{on }\overline\Omega\times[0,\min(T,s_2)].
\end{align*}
\end{itemize}
\end{prop}

\begin{proof}
Let $\varphi$ be a positive solution of \eqref{eq:steady} and $b_c(t)$ be defined in \eqref{eq:seperatet}. If $\lda_1(p-1)<0$, then we know from Proposition \ref{prop:unique} that $\varphi$ is unique. The function $b_c(t)\varphi(x)$ satisfies the equation \eqref{1}. Then the conclusion follows from the comparison principle in Corollary \ref{prop:comparison} by choosing proper $c$ such that either $b_c(0)\varphi(x)\le u_0(x)$ or $b_c(0)\varphi(x)\ge u_0(x)$.
\end{proof}

Next, we show that the equation \eqref{1} has infinite speed of propagation for all $p$ satisfying \eqref{2}, including $0<p<1$.

\begin{proof}[Proof of Theorem \ref{thm:infinitespeed}]
First, it follows from Proposition \ref{cor:boundednessbefore} that there exist $T_0>0$ and $C_0>0$, both of which depend only on $n,p,a,\Omega$ and $\max_{\partial\Omega} u_0$, such that
\begin{equation}\label{eq:infiniteupper}
u\le C_0\quad\mbox{on }(x,t)\in\overline\Omega\times[0,T_0].
\end{equation}

Secondly, we would like to show in the below that there exist $t_0\in (0,T_0]$ and $C_1>0$, both of which depend only on $n,p,a,\Omega$, $\max_{\partial\Omega} u_0$ and $\|u_0\|_{L^p(\partial\Omega)}$, such that
\begin{equation}\label{eq:Lplower}
\int_{\partial\Omega}u^p(\cdot,t)\,\ud S\ge C_1\quad\mbox{for all }t\in [0,t_0].
\end{equation}

The proof of \eqref{eq:Lplower} will be split into several cases. To start with, let $\phi_1$ be the positive eigenfunction associated to $\lda_1$ defined in \eqref{eq:lambda1} such that $\|\phi_1\|_{L^2(\partial\Omega)}=1$. Then by multiplying $\phi_1$ to \eqref{1} and integrating over $\partial\Omega$, we obtain
\begin{equation}\label{eq:Lploweragainst}
\frac{\ud}{\ud t}\int_{\partial\Omega}u^p(\cdot,t)\phi_1\,\ud S=-\lda_1\int_{\partial\Omega}u(\cdot,t)\phi_1\,\ud S.
\end{equation}

Case 1: $\lambda_1\le 0$. Then it follows from \eqref{eq:Lploweragainst} that 
\[
\frac{\ud}{\ud t}\int_{\partial\Omega}u^p(\cdot,t)\phi_1\,\ud S\ge 0,
\]
and thus,
\[
\int_{\partial\Omega}u^p(\cdot,t)\phi_1\,\ud S\ge \int_{\partial\Omega}u^p_0\phi_1\,\ud S\quad\forall\ t\in [0,T_0].
\]
Since $\phi_1$ is bounded from below and above by two positive constants that depend only on $n,p,a$ and $\Omega$, this proves \eqref{eq:Lplower}.

Case 2: $\lambda_1>0$ and $0<p\le 1$. Then from \eqref{eq:infiniteupper} and \eqref{eq:Lploweragainst}, we have
\[
\frac{\ud}{\ud t}\int_{\partial\Omega}u^p(\cdot,t)\phi_1\,\ud S=-\lda_1\int_{\partial\Omega}u(\cdot,t)\phi_1\,\ud S\ge -\lda_1 C_0^{1-p}\int_{\partial\Omega}u^p(\cdot,t)\phi_1\,\ud S.
\]
Solving this differential inequality, we obtain
\[
\int_{\partial\Omega}u^p(\cdot,t)\phi_1\,\ud S \ge e^{-\lda_1 C_0^{1-p} t }\int_{\partial\Omega}u_0^p\phi_1\,\ud S \quad\forall\ t\in [0,T_0].
\]
Hence,
\[
\int_{\partial\Omega}u^p(\cdot,t)\phi_1\,\ud S \ge e^{-\lda_1 C_0^{1-p} T_0}\int_{\partial\Omega}u_0^p\phi_1\,\ud S\quad\forall\ t\in [0,T_0].
\]
Therefore, as in case 1, \eqref{eq:Lplower} follows.

Case 3: $\lambda_1>0$ and $p>1$. Then by using H\"older's inequality,  it follows from \eqref{eq:Lploweragainst} that
\[
\frac{\ud}{\ud t}\int_{\partial\Omega}u^p(\cdot,t)\phi_1\,\ud S=-\lda_1\int_{\partial\Omega}u(\cdot,t)\phi_1\,\ud S\ge -C_2\left(\int_{\partial\Omega}u^p(\cdot,t)\phi_1\,\ud S\right)^{\frac 1p},
\]
where $C_2>0$ depends only on $n,p,a$ and $\Omega$. Since we know from \eqref{eq:Lploweragainst} that the integral $\int_{\partial\Omega}u^p(\cdot,t)\phi_1\,\ud S$ is decreasing in $t$, we obtain
\[
\frac{\ud}{\ud t}\int_{\partial\Omega}u^p(\cdot,t)\phi_1\,\ud S\ge -C_2\left(\int_{\partial\Omega}u^p_0\phi_1\,\ud S\right)^{\frac 1p}.
\]
Integrating in the time variable, we obtain 
\[
\int_{\partial\Omega}u^p(\cdot,t)\phi_1\,\ud S\ge \int_{\partial\Omega}u^p_0\phi_1\,\ud S -t C_2\left(\int_{\partial\Omega}u^p_0\phi_1\,\ud S\right)^{\frac 1p}.
\]
Hence, one proves \eqref{eq:Lplower} by choosing a proper $t_0$.

Finally, we claim that for a sufficiently small $\va_0$, which will be fixed in the end, there holds 
\[
u(x,t)> \va_0 t^{1/p}\quad\mbox{for all }(x,t)\in\overline\Omega\times[0,t_0].
\]
Suppose not, then let $t_1<t_0$ be the first time that $u$ touches the function $\va_0 t^{1/p}$ at some point $x_1\in\partial\Omega$. Since $u_0$ is positive, we have $t_1>0$. Then
\[
\pa_t (u^p) (x_1,t_1)\le \va_0^p\quad\mbox{and}\quad u(x_1,t_1)=\va_0 t_1^{1/p}.
\]
Using the equation \eqref{1}, we then have
\[
-\partial_\nu u(x_1,t_1)\le \va_0^p + \va_0 t_1^{1/p} \|a\|_{L^\infty(\partial\Omega)}  \le \va_0^p +  \va_0 t_0^{1/p} \|a\|_{L^\infty(\partial\Omega)}.
\]
By Lemma \ref{lem:normal derivative lower bound}, then we have
\[
\int_{\partial\Omega}u(\cdot,t_1)\,\ud S\le C_3\va_0^p + C_3 \va_0 t_0^{1/p},
\]
where $C_3>0$ depends only on $n,\Omega$ and $\|a\|_{L^\infty(\partial\Omega)} $. Then when $p\ge 1$, it follows from \eqref{eq:infiniteupper} and \eqref{eq:Lplower} that
\[
C_1\le \int_{\partial\Omega}u^p(\cdot,t_1)\,\ud S \le  C_0^{p-1} \int_{\partial\Omega}u(\cdot,t_1)\,\ud S\le  C_0^{p-1}(C_3\va_0^p + C_3 \va_0 t_0^{1/p}).
\]
This would be impossible if we choose $\va_0$ sufficiently small. When $0<p<1$, then by H\"older's inequality and \eqref{eq:Lplower}, we have 
\[
C_1^{1/p}\le \left(\int_{\partial\Omega}u^p(\cdot,t_1)\,\ud S\right)^{\frac 1p} \le C_4 \int_{\partial\Omega}u(\cdot,t_1)\,\ud S\le C_4(C_3\va_0^p + C_3 \va_0 t_0^{1/p}),
\]
where $C_4>0$ depends only on $n$ and $\Omega$. This would be impossible either, if we choose $\va_0$ sufficiently small.

This theorem is proved.
\end{proof}

\section{Extinction or blow up in finite time, and integral bounds}\label{sec:extinction}

Part (iii) of Proposition \ref{cor:boundednessbefore}  immediately implies that if $\lambda_1 >0$ and $p>1$, then the solution $u$ cannot be positive forever, that is,
\[
\sup\{t>0: u>0\ \  \mbox{on }\ \partial\Omega\times[0,t)\}<+\infty.
\]
Similarly, part (iv) of Proposition \ref{cor:boundednessbefore} implies that if  $\lambda_1 <0$ and $0<p<1$, then the solution $u$ cannot be bounded forever, that is,
\[
\sup\{t>0: \|u\|_{L^\infty(\partial\Omega\times[0,t))}<+\infty\}<+\infty.
\]

\begin{defn}\label{defn:extinction}
Let $u$ be as in Proposition \ref{cor:boundednessbefore}. Suppose $\lda_1 (p-1)>0$. We define
\begin{equation}\label{eq:extinctiontime}
T^*:=
\left\{
  \begin{aligned}
&\sup\{t>0: u>0\ \  \mbox{on }\ \partial\Omega\times[0,t)\}, \quad\mbox{if } \lambda_1 >0\ \mbox{and}\ p>1.   \\
&\sup\{t>0: \|u\|_{L^\infty(\partial\Omega\times[0,t))}<+\infty\}, \quad\mbox{if } \lambda_1 <0\ \mbox{and}\ 0<p<1.
\end{aligned}
\right.
\end{equation}
When $\lambda_1 >0$ and $p>1$, we call $T^*$ as the \emph{extinction time} of $u$. When $\lambda_1 <0$ and $0<p<1$, we call $T^*$ as the \emph{blow-up time} of $u$.
\end{defn}
Consequently, it follows from Proposition \ref{cor:boundednessbefore} that 
\begin{cor}\label{cor:estimateT^*}
Let $T^*$ be as in Definition \ref{defn:extinction}. Then there exist $s_1>0$ small, and $s_2>0$ large, both of which depend only on $n,p,a,\Omega, \max_{\partial\Omega}u_0$ and $\min_{\partial\Omega}u_0$, such that
\[
s_1\le T^*\le s_2.
\]
\end{cor}

The reason why we call this supremum $T^*$ as the extinction time of the solution when $\lambda_1 >0$ and $p>1$, is that if $u$ vanishes at one point at time $T$, then $u(\cdot,T)\equiv 0$, by using the following elliptic type weak Harnack inequality for the problem  \eqref{1}--\eqref{eq:initial}. This idea of proving such a lower bound was used earlier in Jin-Xiong \cite{JX20-2}.

\begin{prop}\label{prop:lower1}
Let $u$ be a smooth positive solution of  \eqref{1}--\eqref{eq:initial} on $\overline\Omega\times[0,T)$ with $p>1$ and $T>0$. Suppose $u\le M$ on $\overline\Omega\times[0,T)$. There exists $C>0$ depending only on $n,p,a,\Omega, M, \min_{\partial\Omega}u_0$ and $\|u_0\|_{C^{1,\alpha}(\partial\Omega)}$ (with any $\alpha>0$ fixed) such that
$$
\inf_{x\in\partial\Omega}u(x,t)\ge \frac1C  \int_{\partial\Omega} u^{p+1}(\cdot,t)  \,\ud S \quad   \forall \   0<t<T.
$$
\end{prop}

\begin{proof}
Define 
\[
H=-p \frac{\partial_t u}{u}.
\]
Then
\[
\Delta(uH)=0\quad\mbox{in }\Omega
\]
and
$$
H=u^{-p}(\partial_\nu u+au)\quad\mbox{on}\quad \partial \Omega.
$$  
Then
\begingroup
\allowdisplaybreaks
\begin{align}
\partial_t H&  =   -pu^{-p-1}(\partial_t u) (\partial_\nu u+au) +u^{-p}(\partial_\nu  +a)\partial_t u   \nonumber \\
                  & =   -u^{-2p}(\partial_t u^{p}) (\partial_\nu u+au)   + \frac{1}{p}u^{-p}(\partial_\nu +a)(u^{-p+1}\partial_t u^{p}) \nonumber\\
                  & =  u^{-2p}(\partial_\nu u+au)^2 - \frac{1}{p}u^{-p}(\partial_\nu+a)(u^{-p+1}  (\partial_\nu+a)u)\nonumber\\
                  & = H^{2} -\frac{1}{p}u^{-p} (\partial_\nu+a) (uH) \nonumber\\
                  & = -\frac{1}{p}u^{-p}\big(  H \partial_\nu u +u \partial_\nu H+ auH \big) +H^{2} \nonumber\\
                  & =-\frac{1}{p}u^{-p+1} \partial_\nu H + \Big(1-\frac{1}{p}\Big) H^{2}\label{eq:auxH}\\
                  & \geq -\frac{1}{p}u^{-p+1} \partial_\nu H.  \nonumber
\end{align}
\endgroup
So we have the linear elliptic equation for $H$:
\begin{equation*}
  \left\{
    \begin{array}{ll}
        \Delta H + 2 u^{-1} \nabla u \cdot \nabla  H = 0                                                                                    \quad  \mbox{in }   \Omega \times(0, T ),\\
\partial_t H  \geq   -\frac{1}{p}u^{-p+1} \partial_\nu H  \quad \mbox{on }  \partial \Omega \times(0, T ).
         \end{array}
  \right.
\end{equation*}
By Proposition \ref{prop:comparison}, we have 
\[
H\geq c_{1}:= \min \{\min_{\partial \Omega}H(\cdot,0),0  \}\quad\mbox{on }\overline\Omega\times[0,T).
\] 
That is
$$
\partial_\nu u + au \geq c_{1}u^{p} \quad \mbox{on }  \partial \Omega\quad\mbox{for all } t\in(0, T ).
$$
So we have for all $t\in [0,T)$ that
\[
  \left\{
    \begin{array}{ll}
        \Delta u  = 0          \quad  \mbox{in }   \Omega,  \\
         \partial_\nu u  \geq  g(x,t) u       \quad \mbox{on }  \partial \Omega 
         \end{array}
  \right.
\]
with $ g(x,t)= c_{1}u^{p-1}-a $ and $|g(x,t)| \leq |c_1|M^{p-1}+|a| $ on $\partial\Omega$. By Lemma \ref{lem:weakharnack}, we have
$$
\inf_{x\in\partial\Omega} u(x,t)\geq  \frac{1}{C}  \int_{\partial\Omega}   u(\cdot,t) \,\ud S\quad\mbox{for all } t\in(0, T ).
$$
Since $u\le M$ on $\overline\Omega\times[0,T)$, we have $\int_{\partial\Omega}   u(\cdot,t) \,\ud S\ge M^{-p}\int_{\partial\Omega}   u^{p+1}(\cdot,t) \,\ud S$, from which the conclusion follows.
\end{proof}

Similarly, if $0<p<1$, and  $\|u(\cdot,t)\|_{L^\infty(\partial\Omega)}$ blows up at time $T^*$, then $\|u(\cdot,t)\|_{L^{p+1}(\partial\Omega)}$ will also blow up at time $T^*$, following from the local maximum principle in Lemma \ref{lem:localmax}.

\begin{prop}\label{prop:upper1}
Let $u$ be a smooth positive solution of  \eqref{1}--\eqref{eq:initial} on $\overline\Omega\times[0,T)$ with $0<p<1$ and $T>0$. Suppose $u\ge m$ on $\overline\Omega\times[0,T)$ for some $m>0$. There exists $C>0$ depending only on $n,p,a,\Omega, m, \max_{\pa\Omega}u_0$ and $\|u_0\|_{C^{1,\alpha}(\partial\Omega)}$ (with any $\alpha>0$ fixed) such that
$$
\sup_{x\in\partial\Omega}u(x,t)\le  C  \int_{\partial\Omega} u^{p+1}(\cdot,t)  \,\ud S \quad   \forall\,  0<t<T.
$$
\end{prop}

\begin{proof}
Let $H$ be the one defined in Proposition \ref{prop:lower1}. Then it follows from \eqref{eq:auxH} that 
\begin{equation*}
  \left\{
    \begin{array}{ll}
        \Delta H + 2 u^{-1} \nabla u \cdot \nabla  H = 0                                                                                    \quad  \mbox{in }   \Omega \times(0, T ),\\
\partial_t H  \leq   -\frac{1}{p}u^{-p+1} \partial_\nu H  \quad \mbox{on }  \partial \Omega \times(0, T ).
         \end{array}
  \right.
\end{equation*}
By Proposition \ref{prop:comparison}, we have 
\[
H\leq c_{2}:= \max \{\max_{\partial \Omega}H(\cdot,0),0  \}\quad\mbox{on }\overline\Omega\times[0,T).
\] 
That is
$$
\partial_\nu u + au \leq c_{2}u^{p} \quad \mbox{on }  \partial \Omega\quad\mbox{for all } t\in(0, T ).
$$
So we have for all $t\in [0,T)$ that
\[
  \left\{
    \begin{array}{ll}
        \Delta u  = 0          \quad  \mbox{in }   \Omega,  \\
         \partial_\nu u  \leq  g(x,t) u       \quad \mbox{on }  \partial \Omega 
         \end{array}
  \right.
\]
with $ g(x,t)= c_{1}u^{p-1}-a $ and $|g(x,t)| \leq |c_1|m^{p-1}+|a| $ on $\partial\Omega$. By Lemma \ref{lem:localmax}, we have
$$
\sup_{x\in\partial\Omega}u(x,t)\leq  C \int_{\partial\Omega}   u(\cdot,t) \,\ud S \quad\mbox{for all } t\in(0, T ).
$$
Since $u\ge m$ on $\overline\Omega\times[0,T)$, we have $\int_{\partial\Omega}   u(\cdot,t) \,\ud S\le m^{-p}\int_{\partial\Omega}   u^{p+1}(\cdot,t) \,\ud S$, from which the conclusion follows.
\end{proof}

Consequently, we have

\begin{prop}\label{prop:integralT^*}
Let $n\ge 2$, $ \Omega \subset \R^n$   be a bounded smooth domain, $p$ satisfy \eqref{2}, $a\in C^\infty(\pa \om)$ and $u_0 \in C^\infty(\pa \om) $ be a  positive function. Suppose $\lambda_1(p-1)>0$, and $u$ is the smooth positive solution of  \eqref{1}--\eqref{eq:initial} on $\overline\Omega\times[0,T^*)$, where $T^*<\infty$ is the one defined in \eqref{eq:extinctiontime}.  
\begin{itemize}
\item[(i).] If $\lambda_1>0$ and $p>1$, then
\begin{equation}\label{eq:extinctionLp}
\liminf_{t\to (T^*)^-} \|u(\cdot,t)\|_{L^{p+1}(\partial\Omega)}=0.
\end{equation}

\item[(ii).] If $\lambda_1<0$ and $0<p<1$, then
\begin{equation}\label{eq:blowupLp}
\limsup_{t\to (T^*)^-} \|u(\cdot,t)\|_{L^{p+1}(\partial\Omega)}=+\infty.
\end{equation}

\end{itemize}
\end{prop}

\begin{proof}
It follows from Proposition \ref{prop:lower1}, Proposition \ref{prop:upper1}, and the definition of $T^*$. \end{proof}

Now we can derive the precise decay rate of $\|u(\cdot,t)\|_{L^{p+1}(\partial\Omega)}$ near the extinction time $T^*$, or the precise blow up rate of $\|u(\cdot,t)\|_{L^{p+1}(\partial\Omega)}$ near the blow up time $T^*$.

\begin{prop}\label{prop:integralbound}
Assume all the assumptions in Proposition \ref{prop:integralT^*}.
Then there exist $C>0$ and $c>0$, both of which depend only on $n,p,a,\Omega, \max_{\partial\Omega}u_0$ and $\min_{\partial\Omega}u_0$, such that
\[
c(T^{\ast} -t)^{\frac{1}{p-1}}\le \|u(\cdot,t)\|_{L^{p+1}(\partial \Omega)}\le C (T^{\ast} -t)^{\frac{1}{p-1}}.
\]
\end{prop}

\begin{proof}
By \eqref{1}, we have
\begin{equation}\label{12}
   \frac{\ud}{\ud t}\int_{\partial \Omega} u^{p+1}(x,t) \,\ud S = -\frac{p+1}{p} \bigg(\int_{\Omega}|\nabla u(x,t)|^{2}\,\ud x +\int_{\partial \Omega}a u^{2}(x,t) \,\ud S \bigg),
\end{equation}
and
\begin{equation*}
    \frac{\ud}{\ud t}\bigg(\int_{\Omega}|\nabla u(x,t)|^{2}\,\ud x + \int_{\partial \Omega} a u^{2}(x,t) \,\ud S \bigg)   =   - \frac{2}{p}  \int_{\partial \Omega} \frac{ (\partial_{\nu}u +a u)^{2}}{u^{p-1} }  \,\ud S\le 0.
\end{equation*}
Let
\[
I(t)=  \frac{  \displaystyle \int_{\Omega}|\nabla u|^{2}\,\ud x    +\int_{\partial \Omega}a u^{2} \,\ud S    }{ \displaystyle\left(\int_{\partial \Omega}u^{p+1} \,\ud S\right)^{\frac{2}{p+1}}   }  .
\]
Then
\begin{align*}
\frac{\ud}{\ud t} I(t)&=\frac{2}{p} \left(\int_{\partial \Omega} u^{p+1}(x,t) \,\ud S\right) ^{-\frac{2}{p+1}}\\
&\cdot \left[\frac{\displaystyle\bigg(\int_{\Omega}|\nabla u(x,t)|^{2}\,\ud x + \int_{\partial \Omega} a u^{2}(x,t) \,\ud S \bigg)^2}{\displaystyle\int_{\partial \Omega} u^{p+1}(x,t) \,\ud S }- \int_{\partial \Omega} \frac{ (\partial_{\nu}u +a u)^{2}}{u^{p-1} }  \,\ud S \right].
\end{align*}
Since $u$ is harmonic in $\Omega$, we have $   \int_{\Omega}|\nabla u|^{2}\,\ud x = \int_{\partial \Omega} u \partial_{\nu}u \,\ud S  $. Applying the Cauchy-Schwarz inequality, we obtain:
\begin{equation*}
  \bigg(\int_{\Omega}|\nabla u|^{2}\,\ud x    +\int_{\partial \Omega}a u^{2} \,\ud S     \bigg)^{2}  \leq  \int_{\partial \Omega}u^{p+1} \,\ud S    \int_{\partial \Omega} \frac{(\partial_{\nu}u +au)^{2}}{u^{p-1}}\,\ud S.
\end{equation*}
Hence,
\begin{equation}\label{14}
\frac{\ud}{\ud t} I(t)\le 0.
\end{equation}

Define
\begin{equation*}
  Z(t):=    \bigg(  \int_{\partial \Omega} u^{p+1}(x,t) \,\ud S \bigg)^{\frac{p-1}{p+1}}.
\end{equation*}
Then it follows from \eqref{12} that
\begin{equation}\label{eq:Zt}
  Z'(t)= -\frac{p-1}{p}  I(t).
\end{equation}

Case 1: $\lambda_1>0$ and $p>1$.

Then it follows from \eqref{eq:Zt} that
\[
  Z'(t)\le -C
\]
for some $C>0$ depending only on $n,p, a$ and $\Omega$, where we used $p>1$, the trace inequality \eqref{thm:trace} and Proposition \ref{prop:sobolev} in the last inequality. Then
\[
Z(T)-Z(t)\le -C(T-t).
\]
Taking $\liminf_{T\to (T^*)^-}$ on both sides and using \eqref{eq:extinctionLp}, we obtain
\begin{equation*}
Z(t)\ge C(T^*-t).
\end{equation*}
Hence, the first inequality follows. 

By \eqref{eq:Zt} and \eqref{14}, we have $Z''(t)\ge 0$. Hence, for $0<s<t<T<T^*$, we have
\[
Z(t)\le Z(T)+\frac{Z(s)-Z(T)}{s-T} (t-T).
\]
Taking $\liminf_{T\to (T^*)^-}$ and $\lim_{s\to 0^+}$ on both sides, and using \eqref{eq:extinctionLp}, we have
\[
Z(t)\le \frac{Z(0)}{T^*} (T^*-t).
\]
Hence, the second inequality follows, with the help of Corollary \ref{cor:estimateT^*}. 

Case 2: $\lambda_1<0$ and $0<p<1$. 

It follows from  \eqref{eq:Zt} and \eqref{eq:finiteyp} that
\begin{equation*}
  Z'(t)= \frac{1-p}{p}  I(t)\ge  \frac{Y_p(1-p)}{p}\ge -C,
\end{equation*}
where $C>0$ depending only on $n,p,\Omega$ and $a$. Then integrating in $t$, we have 
\[
Z(T)-Z(t)\ge -C(T-t).
\]
Taking $\limsup_{T\to (T^*)^-}$ on both sides and using \eqref{eq:blowupLp}, we obtain
\begin{equation*}
Z(t)\le C(T^*-t).
\end{equation*}
Hence, the second inequality follows.

By \eqref{eq:Zt} and \eqref{14}, we have $Z''(t)\le 0$. Hence, for $0<s<t<T<T^*$, we have
\[
Z(t)\ge Z(T)+\frac{Z(s)-Z(T)}{s-T} (t-T).
\]
Taking $\limsup_{T\to (T^*)^-}$ and $\lim_{s\to 0^+}$ on both sides,, and using \eqref{eq:blowupLp}, we have
\[
Z(t)\ge \frac{Z(0)}{T^*} (T^*-t).
\]
Hence, the first inequality follows, with the help of Corollary \ref{cor:estimateT^*}. 
\end{proof}

\section{Uniform upper and lower bounds}\label{sec:uniformbound}

Let $\lambda_1(p-1)>0$, $u$ and $T^*<\infty$ be as in Proposition \ref{prop:integralbound}. 
Define 
\begin{equation}\label{eq:changetow}
  w(x,\tau)=\frac{u(x,t)}{( T^{\ast} - t   )^{\frac{1}{p-1}} } \ \  \mbox{with} \ \ t= T^{\ast}(1-e^{-\tau}).
\end{equation}
Then $w(x,\tau)$ is smooth, positive, locally bounded on $\overline\Omega\times[0,+\infty)$, and satisfies
\begin{equation}\label{18}
\left\{
  \begin{aligned}
 \Delta w(x,\tau) &= 0 \quad  \mbox{in }  \Omega\times(0, \infty ),        \\
     \partial_{\tau}w^p&=-  \partial_{\nu}w  -  a w + \frac{p}{p-1}w^{p}   \quad  \mbox{on }  \partial \Omega \times(0, \infty ),\\
     w(x,0)&=w_{0}:=\frac{u_{0}}{{T^{*}}^{\frac{1}{p-1}}} \quad   \mbox{on }  \partial \Omega \times \{t=0\}.
\end{aligned}
\right.
\end{equation}
It follows from Proposition \ref{prop:integralbound} that there exist $C>0$ and $c>0$, both of which depend only on $n,p,a,\Omega, \max_{\partial\Omega}u_0$ and $\min_{\partial\Omega}u_0$, such that
\begin{equation}\label{19}
c\leq  {\|w(\cdot,\tau)  \|}_{L^{p+1}(\partial \Omega)} \le C\quad\mbox{for all }\tau>0.
\end{equation}
Once we have the integral bounds \eqref{19} of $w$, we can use Moser's iteration, similar to that in Bonforte-V\'azquez \cite{BV} for the classic fast diffusion equation \eqref{eq:fde}, to derive its $L^\infty$ bound. 
\begin{prop}\label{P7-1}
Assume all the assumptions in Proposition \ref{prop:integralbound}. Let $u$ and $T^*$ be as in Proposition \ref{prop:integralbound}. Let $w$ be defined in \eqref{eq:changetow}. Then there exists $C>0 $ depending only on $n,p,a,\Omega, \max_{\partial\Omega}u_0$ and $\min_{\partial\Omega}u_0$ such that
$$
w(x,\tau)\leq C \quad \mbox{for all} \  x  \in   \partial \Omega\mbox{ and }\tau \geq 0.
$$
\end{prop}

\begin{proof}
It follows from Proposition \ref{cor:boundednessbefore} and Corollary \ref{cor:estimateT^*} that there exists $\tau_0>0 $ depending only on $n,p,a,\Omega, \max_{\partial\Omega}u_0$ and $\min_{\partial\Omega}u_0$ such that the conclusion holds for $0\le\tau\le \tau_0$. In the following, we shall prove it for $\tau\ge \tau_0$. By rescaling, we can assume $\tau_0=1$.

We only provide the proof for $n\ge 3$, while the other case $n=2$ can be proved in the same way. 

Let $0<S_{1}<S_{2}<S$, such that $|S_{2}-S_{1}|\leq 1$. Let $\eta(\tau) $ be a smooth cut-off function satisfying  $\eta(\tau)=0$   for  $\tau< S_{1}$,  $0 \leq \eta(\tau)\leq 1$  for   $S_{1}  \leq \tau\leq S_{2}$, $\eta(\tau)=1$  for $\tau>S_{2}$, and $\eta'(\tau) \leq \frac{2}{S_{2}-S_{1}}$.
Let $\Gamma_{1}=\partial \Omega \times [S_{1},S]$, $\Gamma_{2}=\partial \Omega \times [S_{2},S]$, $Q_{1}= \Omega \times [S_{1},S]$ ,            $Q_{2}= \Omega \times [S_{2},S]$.  Let $\alpha> -1$.  In the below, $C$ will be different constants depending only on $n,p,a,\Omega, \max_{\partial\Omega}u_0$ and $\min_{\partial\Omega}u_0$, but may  change from lines to lines.  

Multiplying $\eta^{2}w^{1+\alpha}$ to the boundary equation of \eqref{18} and integrating over $\Gamma_{1}$, we have
$$
\iint_{\Gamma_{1}}\eta^{2}w^{1+\alpha}\frac{\partial}{\partial \tau}w^{p}+\iint_{\Gamma_{1}}\eta^{2}w^{1+\alpha}\frac{\partial}{\partial \nu}w \le -\iint_{\Gamma_{1}}a\eta^{2}w^{2+\alpha}+\frac{p}{|p-1|}\iint_{\Gamma_{1}}\eta^{2}w^{p+1+\alpha}.
$$
Then integration by parts gives
$$
\iint_{\Gamma_{1}}\eta^{2}w^{1+\alpha}\frac{\partial}{\partial \tau}w^{p}+\iint_{Q_{1}}\eta^{2}\nabla(w^{1+\alpha}) \nabla w \leq C\iint_{\Gamma_{1}}\eta^{2}w^{2+\alpha}+C\iint_{\Gamma_{1}}\eta^{2}w^{p+1+\alpha}.
$$
On the left hand side, for the first term, we have
\begingroup
\allowdisplaybreaks
\begin{align*}
    \iint_{\Gamma_{1}}\eta^{2}w^{1+\alpha}\frac{\partial}{\partial \tau}w^{p} & =  \frac{p}{p+1+\alpha}  \iint_{\Gamma_{1}}\eta^{2}\frac{\partial}{\partial \tau}w^{p+1+\alpha}\\
      & = \frac{p}{p+1+\alpha}  \int_{\partial \Omega}(w^{p+1+\alpha})(S)-\frac{p}{p+1+\alpha}  \iint_{\Gamma_{1}}2\eta\eta_{\tau}w^{p+1+\alpha}\\
      &\geq \frac{p}{p+1+\alpha}\int_{\partial \Omega}(w^{p+1+\alpha})(S)- \frac{C}{ S_{2}-S_{1}   } \iint_{\Gamma_{1}}w^{p+1+\alpha}.
\end{align*}
\endgroup
For the second term, we obtain
\begin{equation*}
  \begin{split}
  \iint_{Q_{1}}\eta^{2}\nabla( w^{1+\alpha}) \nabla w    &  =  \frac{4(1+\alpha)}{(2+\alpha)^2}\iint_{Q_{1}}\eta^{2}{|\nabla( w^{\frac{2+\alpha}{2}})|}^{2} .
  \end{split}
\end{equation*}
Using above three estimates, we obtain
\begin{equation}\label{22}
  \int_{\partial \Omega}(w^{p+1+\alpha})(S) + \iint_{Q_{2}}{|\nabla( w^{\frac{2+\alpha}{2}})|}^{2} \leq \frac{C(1+\alpha)}{S_{2}-S_{1}}\iint_{\Gamma_{1}}w^{\alpha}(w^{p+1}+w^{2}).
\end{equation}
We can find a $s_{0}\in [S_{2},S]$ such that
$$
\int_{\partial \Omega}(w^{p+1+\alpha})(s_{0})\geq \frac{1}{2}\sup_{\tau\in{[S_{2},S]}}\int_{\partial \Omega}(w^{p+1+\alpha})(\tau).
$$
We replace $S$ by $s_{0}$ in \eqref{22} to obtain
\begin{equation}\label{23}
  \sup_{\tau\in{[S_{2},S]}}\int_{\partial \Omega}(w^{p+1+\alpha})(\tau)+\iint_{Q_{2}}{|\nabla( w^{\frac{2+\alpha}{2}})|}^{2}\leq
  \frac{C(1+\alpha)}{S_{2}-S_{1}}\iint_{\Gamma_{1}}w^{\alpha}(w^{p+1}+w^{2}).
\end{equation}

Case 1: $p>1$.
By \eqref{19} and H\"older's inequality we get
\begin{equation*}
 {C}^{\frac{p+1}{p-1}}\leq \int_{\partial \Omega}w^{p+1}\leq { |\partial\Omega|}^{\frac{\alpha}{p+1+\alpha}}\Big(\int_{\partial \Omega}w^{p+1+\alpha}\Big)^{\frac{p+1}{p+1+\alpha}}\leq C\Big(\int_{\partial \Omega}w^{p+1+\alpha}\Big)^{\frac{p+1}{p+1+\alpha}}.
\end{equation*}
So we have
\begin{equation}\label{26}
\frac{1}{C}\leq\Big(\int_{\partial \Omega}w^{p+1+\alpha}\Big)^{\frac{1}{p+1+\alpha}}.
\end{equation}
The constant $C$ in \eqref{26} is independent of $\alpha$, because ${ |\partial\Omega|}^{\frac{\alpha}{p+1+\alpha}}$ is bounded by a constant independent of all $   \alpha \geq 0$.
By $p>1$, we have
\begin{equation}\label{27}
\int_{\partial \Omega}w^{\alpha+2}\leq{ |\partial\Omega|}^{\frac{p-1}{p+1+\alpha}}\Big(\int_{\partial \Omega}w^{p+1+\alpha}\Big)^{\frac{\alpha+2}{p+1+\alpha}}\leq C\int_{\partial \Omega}w^{p+1+\alpha}.
\end{equation}
We obtain from  \eqref{23} that
\begin{equation}\label{28}
  \sup_{\tau\in{[S_{2},S]}}\int_{\partial \Omega}(w^{p+1+\alpha})(\tau)+\iint_{Q_{2}}{|\nabla( w^{\frac{2+\alpha}{2}})|}^{2}\leq
  \frac{C(1+\alpha)}{S_{2}-S_{1}}\iint_{\Gamma_{1}}w^{\alpha+p+1}.
\end{equation}
For each $\sigma>1$,  using \eqref{thm:trace} and \eqref{27}, we have
\begin{equation*}
  \begin{split}
        \int_{\partial \Omega} w^{(\alpha+2)\sigma}          &  =  \int_{\partial \Omega} w^{\alpha+2}   w^{(\alpha+2)(\sigma-1)}       \\
                  & \leq     \Big(\int_{\partial \Omega} w^{ \frac{(\alpha+2)(n-1)}{n-2}   } \Big)^{\frac{n-2}{n-1}}  \Big(\int_{\partial \Omega} w^{(\alpha+2)(\sigma-1)(n-1)    } \Big)^{\frac{1}{n-1}}\\
                  &\leq C    \Big(\int_{\Omega}  {|\nabla( w^{\frac{\alpha+2}{2}})|}^{2}  + \int_{\partial \Omega}w^{\alpha+2}   \Big)                     \Big(\int_{\partial \Omega} w^{(\alpha+2)(\sigma-1)(n-1)    } \Big)^{\frac{1}{n-1}}\\
                  &\leq C\Big(\int_{\Omega}  {|\nabla( w^{\frac{\alpha+2}{2}})|}^{2}  +  \int_{\partial \Omega}w^{\alpha+p+1}   \Big)                     \Big(\int_{\partial \Omega} w^{(\alpha+2)(\sigma-1)(n-1)    } \Big)^{\frac{1}{n-1}}.
  \end{split}
\end{equation*}
Integrating the above inequality over time interval $[S_{2},S]$, and using \eqref{27} and \eqref{28}, we obtain 
\begin{equation}\label{29}
  \begin{split}
  &\iint_{\Gamma_{2}} w^{(\alpha+2)\sigma}  \\
    &  \leq C\Big(\iint_{Q_{2}}  {|\nabla( w^{\frac{\alpha+2}{2}})|}^{2}  +  \iint_{\Gamma_{2}}w^{\alpha+p+1}   \Big)    \Big(\sup_{\tau\in[S_{2},S]}\int_{\partial \Omega} w^{(\alpha+2)(\sigma-1)(n-1)    }(\tau) \Big)^{\frac{1}{n-1}}      \\
                  & \leq  \frac{C(1+\alpha)}{S_{2}-S_{1}}\Big(\iint_{\Gamma_{1}}w^{\alpha+p+1}   \Big) \Big(\sup_{\tau\in[S_{2},S]}\int_{\partial \Omega} w^{(\alpha+2)(\sigma-1)(n-1)    }(\tau) \Big)^{\frac{1}{n-1}}  .
  \end{split}
\end{equation}
Let
$$
\alpha+p+1=(\alpha+2)(\sigma-1)(n-1).
$$
Then
$$
\sigma=1+\frac{\alpha+p+1}{(n-1)(\alpha+2)}.
$$
For $p$ satisfying \eqref{2} with $p>1$, i.e. $ 2 < p+1 <   \frac{2(n-1)}{n-2}  $, we know  for all $\alpha \geq 0$ that 
$$
\frac{n}{n-1}<\sigma<\frac{n}{n-2}.
$$
Then by \eqref{28} and \eqref{29}, we have
\begin{equation}\label{30}
  \iint_{\Gamma_{2}}w^{(\alpha+2)\sigma} \leq  C \Big(\frac{1+\alpha}{S_{2}-S_{1}}\Big)^{1+\frac{1}{n-1}}\Big(\iint_{\Gamma_{1}}w^{\alpha+p+1}   \Big)^{1+\frac{1}{n-1}}.
\end{equation}
For $k=0,1,2,3,\dots$, define $\alpha_{0}=0$,  
\begin{align*}
\sigma_{k}&=1+\frac{\alpha_{k}+p+1}{(n-1)(\alpha_{k}+2)}\  \in \ \Big(\frac{n}{n-1},\frac{n}{n-2}\Big) \mbox{ for all} \ \alpha_{k}\geq0,\\
\alpha_{k+1}&=(\alpha_{k}+2)\sigma_{k}-(p+1),\\
q_{k}&=\alpha_{k}+p+1.
\end{align*}
Then
\begin{align*} 
\alpha_{k}&=\Bigg(\Big(1+\frac{1}{n-1}\Big)^{k}-1\Bigg)\Bigg(  p+1-(n-1)(p-1)        \Bigg),\\
q_{k}&=\Big(1+\frac{1}{n-1}\Big)^{k}\Big(  p+1-(n-1)(p-1)        \Big) +(n-1)(p-1).
\end{align*}
Since $p<\frac{n}{n-2}$, one knows that $p+1-(n-1)(p-1)>0$. So $\alpha_{k}$ and $q_{k}$ are strictly increasing and tend to $+\infty$ as $k\to \infty$. Let $s_{0}=\frac{1}{16}$, $s_{1}=c_{0}$, $s_{k+1}-s_{k}=c_{0}k^{-4}$, where $c_{0}= \frac{1}{2}(\sum_{k=0}^{\infty} k^{-4})^{-1}>\frac{3}{16}$. Then
$s_{k}=c_{0}( \sum_{l=0}^{k-1}l^{-4}    )$, $\lim_{k\to\infty}s_{k}=\frac{1}{2}$. Let $\Gamma_{k}=\partial \Omega \times [s_{k},1]$. By \eqref{30} we get
$$
\iint_{\Gamma_{k+1}}w^{q_{k+1}}=\iint_{\Gamma_{k+1}}w^{(\alpha_{k}+2)\sigma_{k}} \leq  C \Big(\frac{1+\alpha_{k}}{c_{0}}k^{4}\Big)^{1+\frac{1}{n-1}}\Big(\iint_{\Gamma_{k}}w^{q_{k}}   \Big)^{1+\frac{1}{n-1}}.
$$
Then
\begin{equation*}
  \begin{split}
&  \Big(\iint_{\Gamma_{k+1}}w^{q_{k+1}}  \Big)^{\frac{1}{q_{k+1}}} \\
   &  \leq C^{ \frac{1}{q_{k+1}}  }\Big(1+\frac{1}{n-1}\Big)^{\frac{2k}{q_{k+1}}(1+\frac{1}{n-1})  } \Big(\iint_{\Gamma_{k}}w^{q_{k}}  \Big)^{\frac{1}{q_{k}}(1+\frac{1}{n-1})\frac{q_{k}}{q_{k+1}}}     \\
                  & \leq C^{ \frac{1}{q_{k+1}}\sum_{j=0}^{k}(1+\frac{1}{n-1} )^{j}  } D ^{ \frac{1}{q_{k+1}}\sum_{j=0}^{k}(k-j)(1+\frac{1}{n-1} )^{j}  }\Big(\iint_{\Gamma_{0}}w^{q_{0}}   \Big)^{ \frac{1}{q_{k+1}} (1+\frac{1}{n-1})^{k+1}},
  \end{split}
\end{equation*}
where $D=(1+\frac{1}{n-1})^{2(1+\frac{1}{n-1})}$. We send $k\to \infty$ to get
$$
{\|w\|}_{L^{\infty}(\partial \Omega \times [\frac{1}{2},1])} \leq C \Big(\iint_{\Gamma_{0}}w^{p+1}   \Big)^{  \big(  p+1-(n-1)(p-1) \big)^{-1}         }\leq C.
$$
Since the equation is translation invariant in the time variable, this proposition for $p>1$ is proved.

Case 2: $0<p<1$.
By \eqref{19} and H\"older's inequality we get
\begin{equation*}
 {C}^{\frac{p+1}{p-1}}\leq \int_{\partial \Omega}w^{p+1}\leq { |\partial\Omega|}^{\frac{1-p+\alpha}{2+\alpha}}\Big(\int_{\partial \Omega}w^{2+\alpha}\Big)^{\frac{p+1}{2+\alpha}}\leq C\Big(\int_{\partial \Omega}w^{2+\alpha}\Big)^{\frac{p+1}{2+\alpha}}.
\end{equation*}
So we have
\begin{equation}\label{26-2}
\frac{1}{C}\leq\Big(\int_{\partial \Omega}w^{2+\alpha}\Big)^{\frac{1}{2+\alpha}}.
\end{equation}
The constant $C$ in \eqref{26-2} is independent of $\alpha$, because ${ |\partial\Omega|}^{\frac{1-p+\alpha}{2+\alpha}}$ is bounded by a constant independent of all $   \alpha \geq p-1$.
By $0<p<1$, we have
\begin{equation}\label{27-2}
\int_{\partial \Omega}w^{p+1+\alpha}\leq{ |\partial\Omega|}^{\frac{1-p}{2+\alpha}}\Big(\int_{\partial \Omega}w^{2+\alpha}\Big)^{\frac{\alpha+p+1}{2+\alpha}}\leq C\int_{\partial \Omega}w^{2+\alpha}.
\end{equation}
We obtain from  \eqref{23} that
\begin{equation}\label{28-2}
  \sup_{\tau\in{[S_{2},S]}}\int_{\partial \Omega}(w^{p+1+\alpha})(\tau)+\iint_{Q_{2}}{|\nabla( w^{\frac{2+\alpha}{2}})|}^{2}\leq
  \frac{C(1+\alpha)}{S_{2}-S_{1}}\iint_{\Gamma_{1}}w^{\alpha+2}.
\end{equation}
For each $\sigma\in (1,\frac{n-1}{n-2})$, i.e. $2\sigma\in (2,\frac{2(n-1)}{n-2})$, using \eqref{thm:trace} and \eqref{27-2}, we have
\begin{equation*}
  \begin{split}
        \int_{\partial \Omega} w^{(\alpha+2)\sigma}          &  =  \int_{\partial \Omega} w^{\alpha+2}   w^{(\alpha+2)(\sigma-1)}       \\
                  & \leq     \Big(\int_{\partial \Omega} w^{ \frac{(\alpha+2)(n-1)}{n-2}   } \Big)^{\frac{n-2}{n-1}}  \Big(\int_{\partial \Omega} w^{(\alpha+2)(\sigma-1)(n-1)    } \Big)^{\frac{1}{n-1}}\\
                  &\leq C    \Big(\int_{\Omega}  {|\nabla( w^{\frac{\alpha+2}{2}})|}^{2}  + \int_{\partial \Omega}w^{\alpha+2}   \Big)                     \Big(\int_{\partial \Omega} w^{(\alpha+2)(\sigma-1)(n-1)    } \Big)^{\frac{1}{n-1}}.
  \end{split}
\end{equation*}
Integrating the above inequality over time interval $[S_{2},S]$, and using \eqref{27-2} and \eqref{28-2}, we obtain 
\begin{equation}\label{29-2}
  \begin{split}
  &\iint_{\Gamma_{2}} w^{(\alpha+2)\sigma}  \\
    &  \leq C\Big(\iint_{Q_{2}}  {|\nabla( w^{\frac{\alpha+2}{2}})|}^{2}  +  \iint_{\Gamma_{2}}w^{\alpha+2}   \Big)    \Big(\sup_{\tau\in[S_{2},S]}\int_{\partial \Omega} w^{(\alpha+2)(\sigma-1)(n-1)    }(\tau) \Big)^{\frac{1}{n-1}}      \\
                  & \leq  \frac{C(1+\alpha)}{S_{2}-S_{1}}\Big(\iint_{\Gamma_{1}}w^{\alpha+2}   \Big) \Big(\sup_{\tau\in[S_{2},S]}\int_{\partial \Omega} w^{(\alpha+2)(\sigma-1)(n-1)    }(\tau) \Big)^{\frac{1}{n-1}}  .
  \end{split}
\end{equation}
Let
$$
\alpha+p+1=(\alpha+2)(\sigma-1)(n-1).
$$
Then
$$
\sigma=1+\frac{\alpha+p+1}{(n-1)(\alpha+2)}.
$$
For $0<p<1$, we know that for all $\alpha \geq p-1$,
$$
1+\frac{2p}{(n-1)(p+1)}<\sigma<\frac{n}{n-1}
$$
Then by \eqref{28-2} and \eqref{29-2}, we have
\begin{equation}\label{30-2}
  \iint_{\Gamma_{2}}w^{(\alpha+2)\sigma} \leq  C \Big(\frac{1+\alpha}{S_{2}-S_{1}}\Big)^{1+\frac{1}{n-1}}\Big(\iint_{\Gamma_{1}}w^{\alpha+2}   \Big)^{1+\frac{1}{n-1}}.
\end{equation}
For $k=0,1,2,3,\dots$, define $\alpha_{0}=p-1$,  
\begin{align*}
\sigma_{k}&=1+\frac{\alpha_{k}+p+1}{(n-1)(\alpha_{k}+2)},\\
\alpha_{k+1}&=(\alpha_{k}+2)\sigma_{k}-2,\\
q_{k}&=\alpha_{k}+2.
\end{align*}
Then
\begin{align*} 
\alpha_{k}&=2p \Big(1+\frac{1}{n-1}\Big)^{k} -p-1,\\
q_{k}&=2p \Big(1+\frac{1}{n-1}\Big)^{k} -p+1.
\end{align*}
So  $q_{k}$ is strictly increasing and tend to $+\infty$ as $k\to \infty$. Let $s_{0}=\frac{1}{16}$, $s_{1}=c_{0}$, $s_{k+1}-s_{k}=c_{0}k^{-4}$, where $c_{0}= \frac{1}{2}(\sum_{k=0}^{\infty} k^{-4})^{-1}>\frac{3}{16}$. Then
$s_{k}=c_{0}( \sum_{l=0}^{k-1}l^{-4}    )$, $\lim_{k\to\infty}s_{k}=\frac{1}{2}$. Let $\Gamma_{k}=\partial \Omega \times [s_{k},1]$. By \eqref{30-2} we get
$$
\iint_{\Gamma_{k+1}}w^{q_{k+1}}=\iint_{\Gamma_{k+1}}w^{(\alpha_{k}+2)\sigma_{k}} \leq  C \Big(\frac{1+\alpha_{k}}{c_{0}}k^{4}\Big)^{1+\frac{1}{n-1}}\Big(\iint_{\Gamma_{k}}w^{q_{k}}   \Big)^{1+\frac{1}{n-1}}.
$$
Then by the similar iteration to that in Case 1, we obtain
$$
{\|w\|}_{L^{\infty}(\partial \Omega \times [\frac{1}{2},1])} \leq C \Big(\iint_{\Gamma_{0}}w^{p+1}   \Big)^{\frac{1}{2p}}\leq C.
$$
Since the equation is translation invariant in the time variable, this proposition for $0<p<1$ is proved.
\end{proof}

Using the uniform upper bound in Proposition \ref{P7-1} and the integral lower bound \eqref{19}, we can obtain the uniform lower bound of $w$, with the help of  Lemma \ref{lem:normal derivative lower bound}.

\begin{prop}\label{P7-2}
Let $w$ be as in Proposition \ref{P7-1}. Then there exists $C>0$, which  depends only on $n,p,a,\Omega, \max_{\partial\Omega}u_0$ and $\min_{\partial\Omega}u_0$,  such that
$$
w(x,\tau)\ge \frac{1}{C}  \quad \mbox{for all} \  x  \in   \partial \Omega\mbox{ and }\tau > 0.
$$
\end{prop}
\begin{proof}

It follows from Proposition \ref{cor:boundednessbefore} and Corollary \ref{cor:estimateT^*} that there exists $\tau_0>0 $ depending only on $n,p,a,\Omega, \max_{\partial\Omega}u_0$ and $\min_{\partial\Omega}u_0$ such that the conclusion holds for $0<\tau\le \tau_0$. 

We claim that there exists some $\va>0$ sufficiently small, which will be determined in the end, such that $w>\va$ on $\pa \Omega$ for all time.

If not, we denote $\tau_1$ as the first time that $w$ touches $\va$ at some point $x_1 \in \pa \Omega$. We can choose $\va$ small such that $\tau_1>\tau_0$. Then 
\[
w(x,\tau_1)\ge\va\quad\mbox{in }\Omega, \quad w(x_1,\tau_1)=\va, \quad \pa_\tau w^p \Big|_{(x_1,\tau_1)} \le 0.
\]
Then it follows from the equation \eqref{18} that 
\[
-\pa_\nu w \Big|_{(x_1,\tau_1)  } \le C(\va + \va^p),
\]
where $C>0$ depends only on $p$ and $\|a\|_{L^\infty(\pa \Omega)}$. By applying Lemma \ref{lem:normal derivative lower bound} to $w(\cdot,\tau_1)-\va$, we obtain
\[
-\pa_\nu w \Big|_{(x_1,\tau_1)  } \ge C  
\int_{\pa \Omega}(w(\cdot,\tau_1)-\va) \,\ud S.
\]
Combining above two inequalities, we have
\[
\int_{\pa \Omega}w(\cdot,\tau_1) \,\ud S \le C(\va + \va^p).
\]
By the uniform upper bound in Proposition \ref{P7-1} and the integral lower bounds \eqref{19}, we have 
\[
\int_{\pa \Omega}w(\cdot,\tau_0) \,\ud S\ge c,
\]
where $c>0$ depends only on $n,p,a,\Omega, \max_{\partial\Omega}u_0$ and $\min_{\partial\Omega}u_0$. This will reach a contradiction if $\va$ is sufficiently small.
\end{proof}

Now we can prove Theorem \ref{thm:wellposedness}.
\begin{proof}[Proof of Theorem \ref{thm:wellposedness}]
The conclusions for the cases of $\lambda_1(p-1)\le 0$ follow from Theorem \ref{thm:short}, Proposition \ref{cor:boundednessbefore} (i) and (ii). The conclusion for the case $\lambda_1(p-1)>0$ follows from Theorem \ref{thm:short}, Proposition \ref{prop:integralbound}, Proposition \ref{P7-1} and Proposition \ref{P7-2}.
\end{proof}

\section{Convergence}\label{sec:convergence}

Throughout this section, we always let $u$ and $T^*$ be as in Theorem \ref{thm:wellposedness}, and $w$ be defined in \eqref{eq:changeofvariable}. We will show the asymptotic behavior of $w$ in this section.

\subsection{Convergence to a steady state with rates}

It follows from Theorem \ref{thm:wellposedness} that $1/C\le w\le C$ on $\overline\Omega\times[0,+\infty)$ for  some $C\ge 1$ which depends only on $n,p,a,\Omega, \max_{\partial\Omega}u_0$ and $\min_{\partial\Omega}u_0$. Then we can obtain higher order regularity estimates for $w$.
\begin{thm}\label{thm:higherestimate}
For every positive integer $k$, there exists a constant $C>0$ depending only on $k, n,p,a,\Omega, \max_{\partial\Omega}u_0$ and $\min_{\partial\Omega}u_0$ such that 
\[
\|w\|_{C^k(\overline\Omega\times[1,+\infty))}\le C.
\]
\end{thm}
\begin{proof}
Let $\tau_0\ge 1$ be arbitrarily fixed. Let $\eta$ be a smooth cut-off such that $\eta=0$ in $[\tau_0-1,\tau_0-\frac{2}{3}]$, $\eta=1$ in $[\tau_0-\frac{1}{2},\infty)$ and $|\eta^{(k)}|\le c_k$ for every positive integer $k$, where $c_k$ is a positive constant depending only on $k$. Let $\tilde w(x,\tau)=\eta(\tau) w(x,\tau)$. Then
\[
 \left\{
\begin{aligned}
 \Delta \tilde w &= 0 \quad  \mbox{in }  \Omega\times(0, \infty ),        \\
pw^{p-1}\partial_{\tau} \tilde w&=-  \partial_{\nu}\tilde w  -  a \tilde w +    \mathrm{sgn}(\lda_1)\frac{p}{|p-1|}w^{p-1}\tilde w  + pw^p\pa_\tau\eta  \quad  \mbox{on }  \partial \Omega \times(0, \infty ).
\end{aligned}
\right.
\]
Applying Lemma \ref{L2} finitely many times and using the Sobolev embedding, we have
\[
\|w\|_{C^k(\overline\Omega\times[\tau_0,\tau_0+1]}\le C.
\]
Since $\tau_0$ is arbitrary, then the conclusion follows.
\end{proof}

We need the following two error estimates for the proof of convergence.

\begin{lem}\label{lem:expgrouthforerror}
Let $\varphi$ be a positive solution of \eqref{eq:steady}, and let $h(x,\tau)=w(x,\tau)-\varphi(x)$. Then there exists $C>0$ depending only on $n,p,a,\Omega, \max_{\partial\Omega}u_0$ and $\min_{\partial\Omega}u_0$ such that
\begin{equation}\label{eq:energyjump}
\|h(\cdot,\tau+t) \|_{L^{2}(\partial \Omega)}\le Ce^{Ct} \|h(\cdot,\tau) \|_{L^{2}(\partial \Omega)}\quad\forall\, \tau>1, \, t>0.
\end{equation}
\end{lem}
\begin{proof}
Consider the parabolic equation of $ w - \varphi$:
\begin{equation}\label{eq:eqforerror0}
pw^{p-1}\partial_{\tau}h=- \mathscr{B} h  -  a h + \mathrm{sgn}(\lda_1) \frac{p}{|p-1|}gh, 
\end{equation}
where
\begin{equation*}
g=\int_0^1p((1-\lambda)\varphi+\lambda w)^{p-1} \,\ud \lambda.
\end{equation*}
Note that $g(x,\tau)$ is a bounded smooth function on $\pa\Omega\times[0,+\infty)$. 
Multiplying $h$ to  both sides of \eqref{eq:eqforerror0} and integrating by parts, and using $1/C \le w \le C$ and $\|\partial_t w\|\le C$ with constant $C > 0$ by means of Theorem \ref{thm:higherestimate}, we obtain
\[
\frac{\ud}{\ud \tau} \int_{\partial\Omega} h^2(\cdot,\tau)w^{p-1}(\cdot,\tau)\,\ud S\le C \int_{\partial\Omega} h^2(\cdot,\tau)w^{p-1}(\cdot,\tau)\,\ud S.
\]
Using Gronwall's inequality and the fact that $1/C \le w \le C$ again, we obtain \eqref{eq:energyjump}.
\end{proof}

\begin{lem}\label{lem:highererror}
Let $\varphi$ be a positive solution of \eqref{eq:steady}, and let $h(x,\tau)=w(x,\tau)-\varphi(x)$. 
Then there exists $C>0$ depending only on $n,p,a,\Omega, \max_{\partial\Omega}u_0$ and $\min_{\partial\Omega}u_0$ such that
\begin{equation}\label{eq:timederivativecontrol}
\|\pa_\tau h(\cdot,  \tau)\|_{C^{3}(\partial \Omega)}+\|h(\cdot,  \tau)\|_{C^{3}(\partial \Omega)}\le C \| h(\cdot,\tau-1) \|_{L^2(\pa \Omega)}\quad\forall\,\tau>1.
\end{equation}
\end{lem}
\begin{proof}
Let $\tau_0>1$ be arbitrarily fixed. Let $\eta$ be a smooth cut-off such that $\eta=0$ in $[\tau_0-1,\tau_0-\frac{2}{3}]$,  $\eta=1$ in $[\tau_0-\frac{1}{2},\infty)$ and $|\eta^{(k)}|\le c_k$ for every positive integer $k$, where $c_k$ is a positive constant depending only on $k$. Let $\tilde h(x,\tau)=\eta(\tau) h(x,\tau)$. Then
\[
pw^{p-1}\partial_{\tau}\tilde h=- \mathscr{B} \tilde h  -  a \tilde h + \mathrm{sgn}(\lda_1) \frac{p}{|p-1|}g \tilde h+ pw^{p-1} h \partial_{\tau}\eta\quad\mbox{on }\pa\Omega\times[\tau_0-1,\infty).
\]
Using Theorem \ref{thm:higherestimate}, applying Lemma \ref{L2} finitely many times and using the Sobolev embedding, we have
\[
\|\pa_\tau h(\cdot,  \tau_0)\|_{C^{3}(\partial \Omega)}+\|h(\cdot,  \tau_0)\|_{C^{3}(\partial \Omega)}\le C \|h\|_{L^2(\partial \Omega\times[\tau_0-\frac{2}{3},\tau_0))}.
\]
By \eqref{eq:energyjump}, we have
\[
\| h(\cdot,\tau) \|_{L^2(\pa \Omega)}^2  \le C \| h(\cdot,\tau_0-1) \|_{L^2(\pa \Omega)}^2\quad\mbox{for all }\tau\in[\tau_0-1,\tau_0].
\]
Combining these two estimates, we obtain
\[
\|\pa_\tau h(\cdot,  \tau_0)\|_{C^{3}(\partial \Omega)}+\|h(\cdot,  \tau_0)\|_{C^{3}(\partial \Omega)}\le C \| h(\cdot,\tau_0-1) \|_{L^2(\pa \Omega)}.
\]
Since $\tau_0$ is arbitrary, then the conclusion follows.
\end{proof}

For $v\in H^{\frac12}(\partial \Omega)$ and $v\ge 0$, let
\begin{equation}\label{34}
  \begin{split}
     G(v) = \int_{\partial \Omega}\left[\frac{1}{2} \big( v\mathscr{B}v  + a v^{2}\big) - \frac{\mathrm{sgn}(\lda_1)p}{|p-1|(p+1)}  v^{p+1} \right]\,\ud S,
  \end{split}
\end{equation}
where $\mathscr{B}$ is the Dirichlet to Neumann map  defined in \eqref{eq:DN}. 
Then
$$
\langle G'(v),\phi \rangle = \int_{\partial \Omega} \left( \mathscr{B}v  + a v - \frac{\mathrm{sgn}(\lda_1)p}{|p-1|}  v^{p}\right) \phi \,\ud S\quad\mbox{for all }  \phi \in H^{\frac12}(\partial \Omega).
$$
Then we have
\begin{equation}\label{35}
\begin{split}
 \frac{\ud }{\ud \tau}G(w(\cdot,\tau)) & =\langle G'(w(\cdot,\tau)),w_{\tau}(\cdot,\tau)\rangle \\
        & =      \int_{\partial \Omega}-\frac{\partial}{\partial \tau}(w^{p})  w_{\tau} \,\ud S \                                       \\
        & = - \int_{\partial \Omega}pw^{p-1} w_{\tau}^{2}\,\ud S =  -\frac{4p}{(p+1)^{2}} \int_{\partial \Omega} { \bigg|\frac{\partial}{\partial \tau}\big(w^{\frac{p+1}{2}}\big)    \bigg|   }^{2}\,\ud S    \leq 0.
\end{split}
\end{equation}
That is, $G(w(\cdot,\tau))$ is decreasing in $\tau$. Moreover, it follows from Theorem \ref{thm:wellposedness} that 
$$
G(w(\cdot,\tau)) \geq -M  \quad \mbox{for all }    \tau\geq0,
$$
where $M>0$ depends only on $n,p,a,\Omega, \max_{\partial\Omega}u_0$ and $\min_{\partial\Omega}u_0$. Consequently,
\begin{equation}\label{eq:Glimit}
\lim_{\tau\to+\infty} G(w(\cdot,\tau))\ \mbox{ exists and is finite.}
\end{equation}

Now we give the proof of Theorem \ref{mainTH}.
\begin{proof}[Proof of Theorem \ref{mainTH}]
By Theorem \ref{thm:higherestimate}, there exists a sequence of times ${\tau_{j}}$, $\tau_{j}\to \infty$ as  $j \to \infty$, such that
$$
w(\cdot,\tau_{j}) \to \varphi  \quad  \mbox{as } j \to \infty\quad\mbox{ in }C^2(\overline\Omega).
$$

First, we will show that $\varphi$ is a solution to \eqref{eq:steady}. Integrating \eqref{35} from $\tau_{j}$ to $\tau_{j}+t$, and using the Cauchy-Schwarz inequality, we obtain
\begin{equation*}
\begin{split}
  \int_{\partial \Omega}{|w^{\frac{p+1}{2}}(\cdot,\tau_{j}+t)  -   w^{\frac{p+1}{2}}(\cdot,\tau_{j}) |}^{2}\,\ud S
        & = \int_{\partial \Omega}   \Big(  \int_{\tau_{j}}^{\tau_{j}+t} \frac{\partial}{\partial \tau}( w^{\frac{p+1}{2}} )\,\ud\tau     \Big)^{2}  \,\ud S  \\
        & \leq \int_{\partial \Omega}  t  \int_{\tau_{j}}^{\tau_{j}+t} { \bigg|\frac{\partial}{\partial \tau}\big(w^{\frac{p+1}{2}}\big)    \bigg|   }^{2}\,\ud\tau \ud S \\
        & = \frac{(p+1)^{2}t}{4p} \bigg(  G(w(\cdot,\tau_j))-  G(w(\cdot,\tau_{j}+t)) \bigg).
\end{split}
\end{equation*}
Then by \eqref{eq:Glimit}, we have
\[
  \int_{\partial \Omega}{|w^{\frac{p+1}{2}}(\cdot,\tau_{j}+t)  -   \varphi^{\frac{p+1}{2}} |}^{2}\,\ud S \to 0
\]
locally uniformly in $t$ as $j\to \infty$. Since $1/C \le w \le C$, we have $1/C \le \varphi \le C$. Then by using the inequality 
\[
|w-\varphi|^{\max(2,p+1)}\le C |w^{\frac{p+1}{2}}-\varphi^{\frac{p+1}{2}}|^2,
\] 
we have
$$
w(\cdot,\tau_{j}+t)\to \varphi \quad\mbox{ in }L^{p+1}(\partial\Omega)
$$
locally uniformly in $t$ as $j\to \infty$. Since they are all uniformly bounded in the $C^3$ norm, then by interpolation inequalities, we have
$$
w(\cdot,\tau_{j}+t)\to \varphi \quad\mbox{ in }C^2(\overline\Omega)
$$
locally uniformly in $t$ as $j\to \infty$. By integrating the equation \eqref{eq:scale} on $\partial \Omega \times [\tau_{j},\tau_{j}+1]$, we have for $\forall$ $\phi$ $\in$ $C^{\infty}(\partial \Omega)$ that
\begin{align*}
&\int_{\partial \Omega} \bigg(w^{p}(\cdot,\tau_{j}+1) - w^{p}(\cdot,\tau_{j})\bigg)\phi \,\ud S  \\
&= \int_{0}^{1}\int_{\partial \Omega} \bigg(-\mathscr{B}w(\cdot,\tau_{j}+t) - a w(\cdot,\tau_{j}+t) +\frac{\mathrm{sgn}(\lda_1)p}{|p-1|} w^{p}(\cdot,\tau_{j}+t)\bigg)\phi \,\ud S\ud t.
\end{align*}
Let $j\to \infty$ we have
$$
\int_{\partial \Omega} \bigg(-\mathscr{B}\varphi - a \varphi +\frac{\mathrm{sgn}(\lda_1)p}{|p-1|} \varphi^{p}                                   \bigg)\phi \,\ud S = 0.
$$
Therefore, $\varphi$ is a stationary solution to \eqref{eq:steady}. 

Secondly, we are going to prove that
$$
w(\cdot,\tau) \to \varphi \quad \mbox{as } \tau \to \infty \quad\mbox{ in }C^2(\overline\Omega).
$$
This follows from the same idea as the uniqueness result of Simon \cite{Simon1983Asymptotics}, and the details are given in the below. Denote
$$
\nabla G(w)= \mathscr{B}w + aw - \frac{\mathrm{sgn}(\lda_1)p}{|p-1|}w^{p}.
$$
Hence,
$$
\nabla G(\varphi) = 0.
$$
We also know that $\varphi \in \mathcal{S}_{C}$ for some $C>0$, where $\mathcal{S}_{C}= \{u \in H^{1/2}(\partial \Omega): \frac{1}{C} \leq u \leq C    \}$, and there exists a neighborhood $U$ of $\varphi$, $U \subset \mathcal{S}_{2C} $ such that $G$ is analytic in $U$. Then by Theorem 3 in Simon \cite{Simon1983Asymptotics} (or Corollary 3.11 in Chill \cite{Chill} or Theorem 4.1 of Haraux-Jendoubi \cite{HJ11} for an abstract setting), we have the following {\L}ojasiewicz type inequality:  There exist $C,\delta_{0}>0 $ and $\theta \in (0,1/2]$ such that for every $w\in C^{2,\alpha}(\partial \Omega)$ with $\|w - \varphi \|_{C^{2,\alpha}(\partial \Omega)}<\delta_{0}$, the following inequality holds.
\begin{equation}\label{40}
\|\nabla G(w)\|_{L^{2}(\partial \Omega)} \geq C| G(w)-G(\varphi) |^{1-\theta}.
\end{equation}
From the equality \eqref{35} and uniform bounds of $w$, we find a constant $c_{0}>0$ such that
\begin{equation}\label{41}
\begin{split}
 -\frac{\ud}{\ud\tau} G(w(\cdot,\tau))&\geq c_0\| (w^p)_{\tau}(\cdot,\tau)  \|_{L^{2}(\partial \Omega)}^2 \\&\ge c_{0}\| w_{\tau}(\cdot,\tau)  \|_{L^{2}(\partial \Omega)} \|  \nabla G(w (\cdot,\tau))    \|_{L^{2}(\partial \Omega)}.
 \end{split}
\end{equation}

For any $\varepsilon \in (0,\delta_{0})$, we shall find a $\tilde{\tau}(\varepsilon)$ such that for any $\tau >\tilde{\tau}(\varepsilon)$,
$$
\|w(\cdot,\tau) - \varphi \|_{C^{2,\alpha}(\partial \Omega)}<\varepsilon .
$$

Let $\tau_{0}$ be large enough such that $    \|w(\cdot, \tau_{0} ) - \varphi \|_{L^{2}(\partial \Omega)}<\delta$, where $\delta$ will be chosen later. Then by Lemma \ref{lem:expgrouthforerror}, 
$$
\|w(\cdot, \tau ) - \varphi \|_{L^{2}(\partial \Omega)}<  C_1\delta \ \ \forall \ \tau \in [\tau_{0},\tau_{0}+1].
$$
Then by Lemma \ref{lem:highererror}, we have
\begin{equation}\label{eq:L2C2jump}
\|w(\cdot, \tau ) - \varphi \|_{C^{2,\alpha}(\partial \Omega)}<  C_1C_2\delta \ \ \forall \ \tau \in [\tau_{0}+1,\tau_{0}+2].
\end{equation}
Choose $\delta$ small so that $C_1C_2\delta<\va$, and 
define
$$
T= \sup \big\{ t :    \|w(\cdot, s ) - \varphi \|_{C^{2,\alpha}(\partial \Omega)}<  \varepsilon  \ \ \forall \ s \in [\tau_{0}+1,t]      \big\}.
$$
Consequently,
$$
T \geq \tau_{0}+2.
$$
We are going to prove that $T$ must be $\infty$. If $T<\infty$, then for any $s \in [\tau_{0}+1,T]$, we have
\begin{equation*}
  \begin{split}
  &-\frac{\ud}{\ud\tau}\bigg( G(w(\cdot,\tau))-G(\varphi)             \bigg)^{\theta} \\
  &= -\theta \bigg( G(w(\cdot,\tau))-G(\varphi)             \bigg)^{\theta-1}\frac{\ud}{\ud\tau}G(w(\cdot,\tau)) \\
    & \geq c_{0}\theta \bigg( G(w(\cdot,\tau))-G(\varphi)             \bigg)^{\theta-1}\| w_{\tau}(\cdot,\tau)  \|_{L^{2}(\partial \Omega)} \|  \nabla G(w (\cdot,\tau))    \|_{L^{2}(\partial \Omega)}\\
    & \geq c_{0}\theta\| w_{\tau}(\cdot,\tau)  \|_{L^{2}(\partial \Omega)},
\end{split}
\end{equation*}
where we used \eqref{40} and \eqref{41}. By integrating the above inequality  from $\tau_{0}+1$ to $s$, we obtain
\begin{equation*}
  \begin{split}
\int_{\tau_{0}+1}^{s}c_{0}\theta\| w_{\tau}(\cdot,\tau)  \|_{L^{2}(\partial \Omega)} \,\ud\tau & \leq \bigg( G(w(\cdot, \tau_{0}+1  ))-G(\varphi)             \bigg)^{\theta}- \bigg( G(w(\cdot,s))-G(\varphi)             \bigg)^{\theta} \\
& \leq \bigg( G(w( \cdot,\tau_{0}+1))-G(\varphi)             \bigg)^{\theta}.
\end{split}
\end{equation*}
By the Minkovski integral inequality, we obtain
\begin{equation*}
  \begin{split}
  \| w(\cdot,s)-  w(\cdot, \tau_{0}+1)  \|_{L^{2}(\partial \Omega)}& = \left\| \int_{\tau_{0}+1}^{s} w_{\tau}(\cdot,\tau) \,\ud\tau  \right\|_{L^{2}(\partial \Omega)}\\
& \leq \int_{\tau_{0}+1}^{s}\| w_{\tau}(\cdot,\tau)  \|_{L^{2}(\partial \Omega)}\,\ud\tau \\
& \leq  \frac{1}{c_{0}\theta}   \bigg( G(w(\cdot, \tau_{0}+1  ))-G(\varphi)             \bigg)^{\theta}.
\end{split}
\end{equation*}
Then
\begin{equation*}
  \begin{split}
  \| w(\cdot,s)-  \varphi  \|_{L^{2}(\partial \Omega)}&\leq \| w(\cdot,s)-  w(\cdot, \tau_{0}+1)  \|_{L^{2}(\partial \Omega)} +\| \varphi-  w(\cdot, \tau_{0}+1)  \|_{L^{2}(\partial \Omega)}\\
  & \leq \frac{1}{c_{0}\theta}   \bigg( G(w( \cdot,\tau_{0}+1 ))-G(\varphi)\bigg)^{\theta}+\| \varphi-  w( \cdot,\tau_{0}+1)  \|_{L^{2}(\partial \Omega)}\\
  &\le C (\| \varphi-  w( \cdot,\tau_{0}+1)  \|_{C^{2,\alpha}(\partial \Omega)}^\theta+\| \varphi-  w( \cdot,\tau_{0}+1)  \|_{C^{2,\alpha}(\partial \Omega)})\\
  &\le C_3\delta^\theta
\end{split}
\end{equation*}
for $\forall s \in [\tau_{0}+1,T]$. If we choose $\delta$ smaller such that $C_2C_3\delta^\theta<\va$, then it follows from \eqref{eq:L2C2jump} that
$$
\| w(\cdot,s)-  \varphi  \|_{C^{2,\alpha}(\partial \Omega)}< \varepsilon \quad \forall  s \in [\tau_{0}+2,T+1] .
$$
This is a contradiction to the definition of $T$. Thus $T=\infty$. Therefore,
$$
\| w(\cdot,\tau)-  \varphi  \|_{C^{2}(\partial \Omega)}   \to \  0\quad\mbox{as }\tau \to \infty.
$$

Finally, we can use the decay rate of $G$ to obtain the decay rate of $w$ in $C^{2,\alpha}(\partial \Omega)$.
By \eqref{35},  \eqref{40} and the uniform bounds of $w$, we have
\begin{equation*}
\begin{split}
  \frac{\ud}{\ud\tau}\bigg( G(w(\cdot,\tau))-G(\varphi)             \bigg) &= -\int_{\partial \Omega}\frac{\partial}{\partial \tau}(w^{p})  w_{\tau} \,\ud S \\
                              &\leq -C \int_{\partial \Omega}\big(\frac{\partial}{\partial \tau}(w^{p}(\cdot,\tau))\big)^{2}   \,\ud S\\
                              &= -C\|\nabla G(w(\cdot,\tau))\|_{L^{2}(\partial \Omega)}^{2}\\
                              &\leq -C| G(w(\cdot,\tau))-G(\varphi) |^{2-2\theta}
\end{split}
\end{equation*}
for $\tau \ge \tau_{0}$. For $\theta \in (0,1/2)$, we know
$$
\frac{\ud}{\ud\tau}\bigg( G(w(\cdot,\tau))-G(\varphi)             \bigg)^{2\theta-1} \geq C(1-2\theta),
$$
and thus,
$$
G(w(\cdot,\tau))-G(\varphi) \leq C \tau^{\frac{1}{2\theta-1}}.
$$
Then for $r$ very large,
\begin{equation}\label{48}
\begin{split}
  \int_{r}^{2r}   \| w_{\tau}(\cdot,\tau)  \|_{L^{2}(\partial \Omega)}         \,\ud\tau  &  \leq  r^{1/2}\bigg(\int_{r}^{2r}   \| w_{\tau}(\cdot,\tau)  \|^{2}_{L^{2}(\partial \Omega)}\bigg)^{1/2} \\
  & \leq C  r^{1/2} \bigg(\int_{r}^{2r}  \int_{\partial \Omega} pw^{p-1}   w_{\tau}^{2}  \,\ud S     \bigg)^{1/2}\\
  & = C  r^{1/2} \bigg(G(w(\cdot,r)-G(w(\cdot,2r))     \bigg)^{1/2}\\
  & \leq C  r^{1/2}\bigg(G(w(\cdot,r)-G(\varphi)     \bigg)^{1/2} \\
  & \leq C  r^{1/2} r^{\frac{1}{4\theta-2}}\\
  & \leq C  r^{\frac{\theta}{2\theta-1}},
\end{split}
\end{equation}
where we used \eqref{35} in the first equality. Then
\begin{equation*}
  \begin{split}
    \int_{r}^{\infty}   \| w_{\tau}(\cdot,\tau)  \|_{L^{2}(\partial \Omega)}         \,\ud\tau    &= \sum_{k=0}^{\infty}  \int_{ r\cdot 2^{k} }^{r \cdot 2^{k+1} }   \| w_{\tau}(\cdot,\tau)  \|_{L^{2}(\partial \Omega)}         \,\ud\tau   \\  &\leq C \sum_{k=0}^{\infty} {(r\cdot 2^{k})}^{\frac{\theta}{2\theta-1}}\\
       &= C \sum_{k=0}^{\infty} {(2^{\frac{\theta}{2\theta-1}})}^{k} r^{\frac{\theta}{2\theta-1}} .
  \end{split}
\end{equation*}
Since $\theta \in (0,1/2)$ , we know $2^{\frac{\theta}{2\theta-1}} <1$, and thus, $\sum_{k=0}^{\infty} {(2^{\frac{\theta}{2\theta-1}})}^{k}$ is finite. So
$$
\int_{r}^{\infty}   \| w_{\tau}(\cdot,\tau)  \|_{L^{2}(\partial \Omega)}         \,\ud\tau \leq C r^{\frac{\theta}{2\theta-1}}.
$$
Finally, for $t$ large,
$$
\|w(\cdot, t ) - \varphi \|_{L^{2}(\partial \Omega)}\leq \left\| \int_{t}^{\infty}w_{\tau}(\cdot,\tau)\,\ud\tau \right\|_{L^{2}(\partial \Omega)}\leq  \int_{t}^{\infty}   \| w_{\tau}(\cdot,\tau)  \|_{L^{2}(\partial \Omega)}         \,\ud\tau \leq Ct^{\frac{\theta}{2\theta-1}}.
$$
Then it follows from Lemma \ref{lem:highererror} that
$$
\|w(\cdot, \tau) - \varphi \|_{C^{2,\alpha}(\partial \Omega)}\leq C\tau^{-\gamma},
$$
where $\gamma= \frac{\theta}{1-2\theta}>0$.

If $\varphi$ is integrable, then by Lemma 1 of Adams-Simon \cite{AS} (see also Corollary 3.12 of Chill \cite{Chill}), the {\L}ojasiewicz inequality \eqref{40} holds for $\theta=\frac12$. That is,
\begin{equation*}
  \begin{split}
G(w(\cdot,\tau))-G(\varphi)&\le C\left\| -\pa_\nu  w(\cdot,\tau) - a w(\cdot,\tau)+\mathrm{sgn}(\lda_1)\frac{p}{|p-1|}{w(\cdot,\tau)}^{p}\right\|^2_{L^2(\pa \om)}\\
&= C{\| \pa_\tau (w^p)       \|}^2_{L^2(\pa \om)}.
  \end{split}
\end{equation*}
Therefore,
\begin{align*}
\frac{\ud}{\ud\tau}\bigg[ G(w(\cdot,\tau))-G(\varphi)\bigg] &= -\int_{\partial \Omega}(w^{p})_{\tau}   w_{\tau} \,\ud S\\
& \le -C{\| \pa_\tau (w^p)\|}^2_{L^2(\pa \om)} \\
&\le -C   (G(w(\cdot,\tau))-G(\varphi)).
\end{align*}
So by Gronwall's inequality, we have
$$
G(w(\cdot,\tau))-G(\varphi) \le Ce^{-4\gamma \tau}
$$
for some $C>0$ and $\gamma>0$. Similar to \eqref{48}, we have
\begin{equation*}
  \int_{r}^{2r}   \| w_{\tau}(\cdot,\tau)  \|_{L^{2}(\partial \Omega)}         \,\ud\tau    \leq C  r^{1/2}\bigg(G(w(\cdot,r)-G(\varphi)     \bigg)^{1/2}  \leq C  r^{1/2} e^{-2\gamma r} \leq C   e^{-\gamma r}
\end{equation*}
for $r$ sufficiently large. Then
\begin{equation*}
    \int_{r}^{\infty}   \| w_{\tau}(\cdot,\tau)  \|_{L^{2}(\partial \Omega)}         \,\ud\tau   = \sum_{k=0}^{\infty}  \int_{ r\cdot 2^{k} }^{r \cdot 2^{k+1} }   \| w_{\tau}(\cdot,\tau)  \|_{L^{2}(\partial \Omega)}         \,\ud\tau     \leq C \sum_{k=0}^{\infty} e^{-\gamma  r \cdot 2^{k}}\leq C e^{-  \gamma  r }
\end{equation*}
for $r$ sufficiently large. Finally, for $t$ large
$$
\|w( \cdot,t ) - \varphi \|_{L^{2}(\partial \Omega)}\leq \left\| \int_{t}^{\infty}w_{\tau}(\cdot,\tau)\,\ud\tau\right\|_{L^{2}(\partial \Omega)}\leq  \int_{t}^{\infty}   \| w_{\tau}(\cdot,\tau)  \|_{L^{2}(\partial \Omega)}         \,\ud\tau \leq Ce^{- \gamma  t}.
$$
Then it follows from Lemma \ref{lem:highererror} that
$$
\|w( \cdot,\tau) - \varphi \|_{C^{2}(\partial \Omega)}\leq Ce^{-\gamma \tau}\quad\mbox{for all }\tau>1.
$$

This finishes the proof of parts (ii), (iv) of Theorem \ref{mainTH}, as well as the upper bound in \eqref{eq:convergencerate3}. 

For part (i) of Theorem \ref{mainTH}, we know  from part (i) of Proposition \ref{cor:boundednessbefore} that
\begin{equation*}
\|w( \cdot,\tau) - \varphi \|_{L^{\infty}(\partial \Omega)}\leq Ce^{- \tau}\quad\mbox{for all }\tau>1.
\end{equation*}
It follows from Lemma \ref{lem:highererror} that
$$
\|w( \cdot,\tau) - \varphi \|_{C^{2}(\partial \Omega)}\leq Ce^{- \tau}\quad\mbox{for all }\tau>1.
$$
This finishes the proof of part (i) of Theorem \ref{mainTH}. 

The dichotomy between \eqref{eq:convergencerate4} and the lower bound in \eqref{eq:convergencerate3} follows from Theorem \ref{either fast or slower} in the next subsection.
\end{proof}

\subsection{Sharp rates and higher order asymptotics}
In this subsection, we adapt the argument of Choi-McCann-Seis \cite{CMS} to show that the solutions of \eqref{eq:scale} will converge to steady states either exponentially with the sharp rate or algebraically slow with the rate $\tau^{-1}$, and also to obtain higher order asymptotics.

Let us assume all the assumptions in Theorem \ref{mainTH}, and $w(\cdot,\tau)\to\varphi$ in $C^2(\partial\Omega)$, where $\varphi$ is a positive smooth solution of \eqref{eq:steady}. Let $\mathcal{L}_{\varphi}$ be defined in \eqref{eq:linearized}. Then  the weighted eigenvalue problem
\be \label{eigenproblem}
\mathcal{L}_{\varphi} (e)  =\mu \varphi^{p-1}e   \quad  \mbox{on }   \partial \Omega
\ee
admits eigenpairs $\{ (\mu_j,e_j)       \}_{j=1}^{\infty}$ such that

\begin{itemize}
           \item  the eigenvalues with multiplicities can be listed as $\mu_1<\mu_2\le \mu_3 \le \cdots \le \mu_j \to +\infty$ as  $j \to +\infty$,
           \item the eigenfunctions $\{ e_j      \}_{j=1}^{\infty}$ forms a complete orthonormal basis of  $L^2(\pa \om; \varphi^{p-1}\,\ud S) $, that is, $\int_{\partial \Omega} e_i e_j \varphi^{p-1}  \,\ud S =  \delta_{ij}$ for $i,j \in \mathbb{N} $ .
\end{itemize}
Since $\varphi$ is a positive solution to \eqref{eq:steady}, then 
\[
\mu_1=- \mathrm{sgn}(\lda_1(p-1)) p\quad\mbox{and}\quad e_1=\varphi / \| \varphi  \|_{L^2(\pa \om; \varphi^{p-1}\,\ud S)}. 
\]

We suppose $I$ is the dimension of the eigenspace corresponding to all the negative eigenvalues, if negative eigenvalues exist, and we denote $\mu_I$ as the largest negative eigenvalue. If $\mu_1\ge 0$, then there are no negative eigenvalues, and we just let $I=0$. Let $K$ be the multiplicity of the zero eigenvalue.  Let $k=I+K+1$ so that $\mu_k$ is the smallest positive eigenvalue. For example, if $\mu_1<0$, then  we can list the eigenvalues with multiplicities as
\[
- \mathrm{sgn}(\lda_1(p-1)) p=\mu_1<\mu_2\le \cdots \le \mu_I <0= \mu_{I+1}=\cdots =\mu_{I+K}<\mu_k\le \cdots.
\] 
Denote
\be \label{sharprate}
\gamma_p = \frac{\mu_k }{p}.
\ee
In particular, if $\lambda_1(p-1)<0$, then $k=1$, $\mu_1=p$, $e_1=\varphi$, and thus, $\gamma_p=1$.

As in \cite{CMS}, we call the eigenfunctions corresponding to the negative eigenvalues the unstable modes, those  corresponding to $\mu_{I+1}$ to $\mu_{I+K}$ the central modes and the remaining eigenfunctions the stable modes. Their corresponding eigenspaces are denoted as $E_u$, $E_c$ and $E_s$. Then we have $ L^2(\pa \om; \varphi^{p-1}\,\ud S) = E_u\oplus E_c\oplus E_s $.  

As before, we let $h=w-\varphi$. Then by  \eqref{eq:stationary} and \eqref{eq:scale},  $h$ satisfies
\begin{equation}\label{eq:error}
p\varphi^{p-1}\pa_\tau h + \mathcal{L}_{\varphi} h = N(h),
\end{equation}
where $\mathcal{L}_{\varphi}$ is the linearized operator defined in \eqref{eq:linearized} and $N(h)$ is the nonlinearity
\begin{equation}\label{eq:nonlinearity}
\begin{split}
& N(h)= \mathrm{sgn}(\lda_1)\frac{p}{|p-1|}\left(\left(h+\varphi\right)^{p}-\varphi^{p}-ph\varphi^{p-1}\right)+p\left(\varphi^{p-1}- \left(h+\varphi\right)^{p-1} \right)\pa_\tau h .
\end{split}
\end{equation}

\begin{thm}\label{either fast or slower}
Assume the assumptions in Theorem \ref{mainTH}. Suppose $w(x,\tau)$ converges to a positive solution $\varphi$ of \eqref{eq:steady} in $C^{2}{(\overline\om)}$ as $\tau\to \infty$. Then $h:=w-\varphi$ satisfies the following dichotomy for some  $C>0$ depending  only on $n,p,a,\Omega, \max_{\partial\Omega}u_0$ and $\min_{\partial\Omega}u_0$:
\begin{itemize}
\item[(a)]
either $h$ decays algebraically or slower, that is,
\[
\| h(\cdot,\tau) \|_{L^2(\pa \Omega)} \ge C\tau^{-1}  \quad \forall\ \tau >1;
\]
\item[(b)]
or $h$ decays exponentially or faster, that is,
\[
\| h(\cdot,\tau) \|_{C^2(\pa \Omega)} \le Ce^{-\gamma_p \tau}    \quad  \forall\ \tau >1, 
\]
where $\gamma_p>0$ is defined in \eqref{sharprate}.
\end{itemize}
\end{thm}

\begin{proof}
Let $P_u$, $P_c$ and $P_s$ be the orthogonal projections of $L^2(\pa \om; \varphi^{p-1}\,\ud S)$ onto the eigen-spaces $E_u$, $E_c$ and $E_s$, respectively. From now on, for the sake of convenience we use
 \[
 \|f\|:=\| f  \|_{L^2(\pa \om; \varphi^{p-1}\,\ud S)}, \quad \langle f,g \rangle := \int_{\pa \Omega} fg {\varphi}^{p-1} \,\ud S,
 \]
and $h_u$, $h_c$, $h_s$ denote the projected solutions $P_u h$, $P_c h$, $P_s h$, respectively. Then we have 
\begin{align}
&\frac{\ud}{\ud\tau}\|h_u\| \ge -\frac{\mu_I}{p}\|h_u\|-\frac{1}{p}\|N(h)\|,\label{eq:error-CHHrefinement11}\\
&\left| \frac{\ud}{\ud\tau}\|h_c\| \right| \le \frac{1}{p}\|N(h)\|,\label{eq:error-CHHrefinement22}\\
&\frac{\ud}{\ud\tau}\|h_s\| \le -\frac{\mu_k}{p}\|h_s\|+\frac{1}{p}\|N(h)\|,\label{eq:error-CHHrefinement33}
\end{align}
which are obtained by multiplying corresponding eigenfunctions to \eqref{eq:error} and integrating.

Next, we recall from Theorem \ref{mainTH} that 
\[
\|h(\cdot, \tau) \|_{C^{2}(\partial \Omega)}=o(1)  \mbox{ as }   \tau \to \infty.
\]
Hence, it follows from \eqref{eq:timederivativecontrol} that for any small $\epsilon>0$, there exists a time $\tau_0(\epsilon)$ such that  $|h|,|\pa_\tau h| \le \epsilon$    for all $\tau \ge \tau_0(\epsilon)$. By the Taylor expansion, we derive 
\begin{equation}\label{eq:nonlinearity-Taylor}
|N(h)|\le C|h|\left(|h|+|\pa_\tau h|\right).
\end{equation}
Therefore, 
\begin{equation}\label{eq:nonlinearity-highorder of h}
\|N(h(\cdot, \tau))\|\le \epsilon\|h(\cdot, \tau)\| \mbox{ for } \tau \ge \tau_0(\epsilon)  .
\end{equation}
Plugging \eqref{eq:nonlinearity-highorder of h} into \eqref{eq:error-CHHrefinement11}, \eqref{eq:error-CHHrefinement22} and \eqref{eq:error-CHHrefinement33}, and using the change of variables $s=\lambda \tau$, we obtain
\begin{equation}\label{eq:error-CHHrefinement2}
\begin{split}
\frac{\ud}{\ud s}\|h_u\|  -\frac{|\mu_I|}{p\lambda}\|h_u\|&\ge-\frac{\epsilon}{p\lambda}\|h\|,\\
\left| \frac{\ud}{\ud s}\|h_c\| \right| &\le \frac{\epsilon}{p\lambda}\|h\|,\\
\frac{\ud}{\ud s}\|h_s\| +\frac{\mu_k}{p\lambda}\|h_s\|&\le \frac{\epsilon}{p\lambda}\|h\|,
\end{split}
\end{equation}
where
\[
\lambda=\left\{
  \begin{aligned}
  & \frac{\mu_k}{2p}    \quad  \mbox{if }   I=0\ (\mbox{i.e., there are no negative eigenvalues}), \\
  &  \frac{1}{2p}\min\{|\mu_I|,\mu_k\}   \quad   \mbox{if }    I\neq 0.
\end{aligned}
\right.
\]
Then $h_u$, $h_c$, $h_s$ satisfy the hypothesis in Lemma 4.6 of Choi-Haslhofer-Hershkovits \cite{CHHrefine} (which is a slight refinement of Lemma A.1 of Merle-Zaag \cite{MZ}). Therefore, either 
\be \label{eq:central dominates}
\| h_u \|+\| h_s \| = o(\| h_c \|) \mbox{ as } \tau \to \infty
\ee
or
\be \label{eq:stable dominates}
\| h_u \|+\| h_c \| \le \frac{100\epsilon}{\lambda p } \| h_s \| \mbox{ for } \tau \ge \tau_0(\epsilon)  
\ee
except that $h=0$.
We investigate these two alternatives separately.

\textbf{Case 1. } If  \eqref{eq:central dominates} holds, we can assume that 
\[
\| h_u \|+\| h_s \| \le \frac{1}{2} \| h_c \|  \mbox{ for } \tau \gg 1.
\]
Then by \eqref{eq:nonlinearity-highorder of h} and \eqref{eq:error-CHHrefinement22}, we have
\[
\frac{\ud}{\ud\tau}\|h_c\|  \ge -\frac{3\epsilon}{2p}\|h_c\|.
\]
Hence,
\[
\|h(\cdot, \tau)\|\le 2 \|h_c(\cdot, \tau)\|\le C  \|h_c(\cdot, \tau+1)\|\le C  \|h(\cdot, \tau+1)\|.
\]
Then by  \eqref{eq:timederivativecontrol} and \eqref{eq:nonlinearity-Taylor}, we drive
\begin{equation*}
\begin{split}
\|N(h(\cdot, \tau))\|  & \le C\|h(\cdot, \tau)\|\left(\|h(\cdot, \tau)\|+\|\pa_\tau h(\cdot, \tau)\|\right)\\
& \le C\|h(\cdot, \tau)\|\left(\|h(\cdot, \tau)\|+\|h(\cdot, \tau-1)\|\right) \\
&\le  C\|h(\cdot, \tau)\|^2\\
& \le C\|h_c(\cdot, \tau)\|^2 \mbox{ for } \tau \gg1.
\end{split}
\end{equation*}
Consequently, \eqref{eq:error-CHHrefinement22} becomes
\[
\frac{\ud}{\ud\tau}\|h_c\|  \ge -C\|h_c\|^2,
\]
which leads to a lower bound
\[
\|h(\cdot, \tau)\| \ge \|h_c(\cdot, \tau)\| \ge (C\tau)^{-1} \mbox{ for } \tau \gg 1.
\] This proves the first alternative in Theorem \ref{either fast or slower}.

\textbf{Case 2. }
If  \eqref{eq:stable dominates} holds, then
we can assume that 
\[
\| h_u \|+\| h_c \| \le \frac{1}{2} \| h_s \|  \mbox{ for } \tau \gg 1.
\]
Then by \eqref{eq:nonlinearity-highorder of h} we obtain
\[
\|N(h(\cdot, \tau))\|\le 2\epsilon\|h_s(\cdot, \tau)\|  \mbox{ for } \tau \ge \tau_0(\epsilon)\gg 1.
\]
Consequently, \eqref{eq:error-CHHrefinement33} becomes
\[
\frac{\ud}{\ud\tau}\|h_s\|  \le -\frac{1}{p}(\mu_k - 2\epsilon)\|h_s\|, 
\]
which, by Gronwall's inequality, leads to an exponential decay
\[
\|h(\cdot, \tau)\| \le 2\|h_s(\cdot, \tau)\| \le Ce^{-\frac{1}{p}(\mu_k - 2\epsilon)(\tau-\tau_0)  }\|h_s(\tau_0)\|   \mbox{ for } \tau \gg 1.
\]
By  \eqref{eq:timederivativecontrol} and \eqref{eq:nonlinearity-Taylor}, we have
\begin{equation*}
\begin{split}
\|N(h(\cdot, \tau))\|  &  \le C\|h(\cdot, \tau)\|\left(\|h(\cdot, \tau)\|+\|\pa_\tau h(\cdot, \tau)\|\right)\\
& \le C\|h(\cdot, \tau)\|\left(\|h(\cdot, \tau)\|+\|h(\cdot, \tau-1)\|\right) \\
 & \le C e^{-\frac{2}{p}(\mu_k - 2\epsilon)(\tau-\tau_0)   }\|h(\cdot, \tau_0)\|^2  \mbox{ for } \tau \gg 1  .
\end{split}
\end{equation*}
Plugging  this into \eqref{eq:error-CHHrefinement33} and using Gronwall's argument again, it yields
\[
\| h(\cdot, \tau)  \| \le Ce^{-\frac{\mu_k}{p}\tau }\|h(\cdot, \tau_0)\| \mbox{ for } \tau \gg1  .
\]
Then part $(b)$ in Theorem \ref{either fast or slower} follows from \eqref{eq:timederivativecontrol}. 

This finishes the proof of Theorem \ref{either fast or slower}.
\end{proof}

In fact, if part $(b)$ in Theorem \ref{either fast or slower} happens, then one can obtain higher order expansions of the solution as follows. First of all, the estimate of the nonlinearity improves as
\be \label{eq:nonlinearity exponetial improvement}
\|N(h(\cdot, \tau))\| \le Ce^{-2\gamma_p \tau} \| h(\cdot, \tau_0)  \|^2\quad\mbox{for all }  \tau\gg 1.
\ee
For $i \in \mathbb{N}$, let $y_i(\tau):= \langle h(\cdot, \tau),e_i \rangle$. Then these projections satisfy
\[
\left|\frac{\ud}{\ud\tau}y_i + \frac{\mu_i}{p}y_i \right| \le \|N(h)\|.
\]
Using \eqref{eq:nonlinearity exponetial improvement} we get
\[
\left|\frac{\ud}{\ud\tau}\left( e^{\frac{\mu_i}{p}\tau} y_i   \right) \right| \le Ce^{-(2\gamma_p-\frac{\mu_i}{p})\tau}.
\]
Let $J\in \mathbb{N}$ is chosen such that $\frac{\mu_{I+J}}{p}< 2\gamma_p\le \frac{\mu_{I+J+1}}{p}$. Then $J\ge K+1$. For $i\in \{I+K+1,  \cdots, I+J \}$, by integrating the above inequality in the  time variable, we obtain 
\[
\left|   e^{\frac{\mu_i}{p}\tau}y_i(\tau)   -  e^{\frac{\mu_i}{p}T}y_i(T)       \right| \le Ce^{-(2\gamma_p-\frac{\mu_i}{p})\tau}
\]
for any $T \ge \tau \gg 1$. This inequality tells us that $\tau \mapsto e^{\frac{\mu_i}{p}\tau}y_i(\tau)$ is a Cauchy sequence, and therefore, 
\[
e^{\frac{\mu_i}{p}\tau}y_i(\tau) \to C_i \in \R  \text{ as } \tau \to \infty
\]
for $i\in \{I+K+1, \cdots, I+J \}$. We rewrite this condition as 
\be \label{eq:decayrate of projection y_i}
y_i(\tau) = C_i e^{-\frac{\mu_i}{p}\tau} + O(e^{-2\gamma_p \tau}  ).
\ee
Next, with this decay in mind, we have
\begingroup
\allowdisplaybreaks
\begin{align}
 &\left\|h-\sum_{i=I+K+1}^{I+J} C_i e^{-\frac{\mu_i}{p} \tau} e_i\right\| \nonumber\\
 &\le \left\|h_s-\sum_{i=I+K+1}^{I+J} C_i e^{-\frac{\mu_i}{p} \tau} e_i\right\|+\| h_u \|+\| h_c \|  \nonumber\\
 &\le\left\|\sum_{i=I+K+1}^{I+J} \left( y_i- C_i e^{-\frac{\mu_i}{p} \tau} \right) e_i\right\| +\left\|h_s-\sum_{i=I+K+1}^{I+J} y_i e_i\right\|+\| h_u \|+\| h_c \| \nonumber\\
 &\le  \sum_{i=I+K+1}^{I+J} \left| y_i- C_i e^{-\frac{\mu_i}{p} \tau} \right| +\left\|h_s-\sum_{i=I+K+1}^{I+J} y_i e_i\right\|+\| h_u \|+\| h_c \|. \label{eq:error-decomposition-leading term}
\end{align}
\endgroup
The first term on the right hand side of \eqref{eq:error-decomposition-leading term} is bounded by $Ce^{-2\gamma_p \tau}$ as we showed in \eqref{eq:decayrate of projection y_i}. For the second term on the right hand side of \eqref{eq:error-decomposition-leading term}, we let
$$
z:=\left\|h_s-\sum_{i=I+K+1}^{I+J} y_i e_i\right\| = \left\|\sum_{i=I+J+1}^{\infty} y_i e_i\right\|. 
$$
Then it satisfies
$$
\frac{\ud}{\ud\tau}z+\frac{\mu_{I+J+1}}{p}z\le C\|N(h)  \| \le Ce^{-2\gamma_p\tau},
$$
which can be obtained in a similar way to that of  \eqref{eq:error-CHHrefinement33}. Then it becomes
$$
\frac{\ud}{\ud\tau}\left( e^{\frac{\mu_{I+J+1}}{p}\tau}z   \right) \le Ce^{-(2\gamma_p-\frac{\mu_{I+J+1}}{p})\tau}.
$$
By Gronwall's argument again, we have
\begin{equation*} 
z(\tau)  \leq C\|h(\tau_0)\| \times \begin{cases}e^{-2 \gamma_p \tau} & \text { if } \mu_{I+J+1}>2 \gamma_p p, \\ \tau e^{-2 \gamma_p \tau} & \text { if } \mu_{I+J+1}= 2 \gamma_p p, 
\end{cases} \quad \tau \gg 1.
\end{equation*}
For the remaining two terms on the right hand side of \eqref{eq:error-decomposition-leading term}, by using \eqref{eq:error-CHHrefinement11} and \eqref{eq:error-CHHrefinement22}, we obtain a differential inequality
$$
\frac{\ud}{\ud\tau}(\| h_u \|+\|h_c\|) \ge -C e^{-2\gamma_p \tau} \mbox{ for } \tau \gg 1  .
$$
Then an integration from $\tau$ to $\infty$ leads to 
$$
\| h_u \|+\|h_c\| \le C e^{-2\gamma_p \tau},
$$
with the help of the alternative $(b)$ that 
$$
\| h_u \|+\|h_c\| \to 0 \mbox{ as } \tau \to \infty.
$$
Finally, plugging the above estimates into \eqref{eq:error-decomposition-leading term}, we obtain  \be \label{eq:error-exp-decay-rate-classify}
\left\|h(\cdot,\tau)-\sum_{i=I+K+1}^{I+J} C_i e^{-\frac{\mu_i}{p} \tau} e_i\right\|\leq  \begin{cases} 
Ce^{-2 \gamma_p \tau} & \text { if } \mu_{I+J+1}>2 \gamma_p p, \\  
C\tau e^{-2 \gamma_p \tau} & \text { if } \mu_{I+J+1}= 2 \gamma_p p, 
\end{cases} 
\ee
for all $\tau \gg 1$ and some $C=C(n,p,a,\Omega,\varphi)$. 

Finally, we can keep expanding the solution up to an arbitrary order in the same way as that Han-Li-Li \cite{HLL} did for the singular Yamabe equation. In the final expansion, the exponential exponents are not only the $\{\mu_j/p\}$ but also some of their linear combinations.

\bigskip

\noindent T. Jin

\noindent Department of Mathematics, The Hong Kong University of Science and Technology\\
Clear Water Bay, Kowloon, Hong Kong\\
Email: \textsf{tianlingjin@ust.hk}

\medskip

\noindent J. Xiong

\noindent School of Mathematical Sciences, Laboratory of Mathematics and Complex Systems, MOE\\
Beijing Normal University, Beijing 100875, China\\
Email: \textsf{ jx@bnu.edu.cn}

\medskip

\noindent X. Yang

\noindent Department of Mathematics, The Hong Kong University of Science and Technology\\
Clear Water Bay, Kowloon, Hong Kong\\
Email: \textsf{maxzyang@ust.hk}

\end{document}